\documentclass{siamart171218}
\usepackage{amssymb,amsmath,amsfonts,amscd}
\usepackage{algorithm,algorithmic}
\usepackage{upgreek,comment,color,url}
\usepackage{epsfig}
\usepackage{graphicx,comment}
\usepackage{multirow}
\usepackage{fancybox}
\usepackage[textsize=tiny]{todonotes}

\usepackage{mystylefile}   % our style file
\usepackage{xfrac}
\headers{Robust Parameter Identifiability Analysis}{Pearce, Ipsen, Haider, Saibaba, Smith}

\title{Robust Parameter Identifiability Analysis  via Column Subset Selection\thanks{Submitted to the editors May 2022.
This research was supported in part by grants 
NSF DMS-1745654 and DOE DE-SC0022085 (ICFI); 
NSF DMS-1845406 (AKS), NSF-DMS-1638521 (KJP, MAH), NSF-DMS-2053812 (MAH, RCS), AFOSR FA9550-18-1-0457 (RCS).  
Address for all authors: Department of Mathematics, North Carolina State University, Raleigh, NC 27695-8205, USA}.}

\author{Katherine J. Pearce\thanks{\email{kjpearce@ncsu.edu}, \url{https://kjpearce.github.io/}.}
\and Ilse C.F. Ipsen\thanks{\email{ipsen@ncsu.edu},
\url{https://ipsen.math.ncsu.edu/}.}
\and Mansoor A. Haider\thanks{\email{mahaider@ncsu.edu},
\url{https://haider.wordpress.ncsu.edu}.}
\and Arvind K. Saibaba\thanks{\email{asaibab@ncsu.edu}, \url{https://asaibab.math.ncsu.edu/}.}
\and Ralph C. Smith\thanks{\email{rsmith@ncsu.edu}, \url{https://https://rsmith.math.ncsu.edu/}.}
}
\begin{document}
\maketitle

\begin{abstract}
We advocate a numerically reliable and accurate approach for practical
parameter identifiability analysis: Applying column subset selection (CSS) to the sensitivity matrix, instead of computing an eigenvalue decomposition of the Fischer information matrix. 
Identifiability analysis via CSS
has three advantages:
(i) It quantifies reliability of the subsets of parameters selected as identifiable and unidentifiable.
(ii) It establishes criteria for comparing the accuracy
of different algorithms.
(iii) The implementations are numerically more accurate and reliable than eigenvalue methods applied to the
Fischer matrix,
yet without an increase in computational cost. The effectiveness of the CSS methods is illustrated with extensive numerical experiments 
on sensitivity  matrices from six physical models, as well as on adversarial synthetic matrices.
Among the CSS methods,
we recommend an implementation based on the strong rank-revealing QR algorithm  because of its rigorous accuracy guarantees for both identifiable and 
non-identifiable parameters.
  \end{abstract}


% REQUIRED

%%%% NOTE: latest file versions are always the ones without version numbers. old versions are numbered chronologically 1,...,k (1 being oldest, k being the previous most-recent version currently being edited)

\section{Introduction}
In data-driven mathematical modeling, the ability to reliably estimate model parameters depends on the set of available observations, the scope of system responses for which such observations are available, the inherent mathematical structure of the  model, and the parameter estimation method.
Identifiability analysis evaluates the ability to accurately estimate each parameter in a model and, in some cases, quantifies the extent to which this estimate is reliable. 
It has wide-ranging implications for a variety of applications, including analysis of disease and epidemiology models to guide treatment regimes, physiologically-based pharmacokinetic (PBPK) and quantitative system pharmacology (QSP) models for drug development, and coupled multi-physics models for next-generation nuclear power plant design. In particular, identifiability analysis can be more challenging, yet also have greater impact, in applications where the number of model variables and  parameters is significantly greater than  the number of responses with available data.

Practical identifiability analysis refers to the partitioning of parameters in a mathematical model into two groups:
{\em identifiable} parameters that can be reliably  estimated from data and those that cannot, termed  {\em unidentifiable}.
At the heart of many practical identifiability methods is the sensitivity matrix $\mchi$, whose columns 
represent  model parameters and whose rows
represent observations (data) for a quantity of interest.
A common approach 
extracts identifiable and unidentifiable 
parameters from eigenvalues and eigenvectors of the Fischer information matrix $\mchi^T \mchi$.
However, the sensitivity matrix $\mchi$ is often ill-conditioned,
that is, sensitive to small perturbations, so that the explicit
formation of the cross product  $\mchi^T \mchi$ can 
inflict a serious loss of accuracy.

We  apply instead column subset selection (CSS)
to the sensitivity matrix~$\mchi$, which has the same 
computational complexity as eigenvalue methods on  the Fischer
matrix $\mchi^T \mchi$.
We derive bounds that show the 
superior accuracy of CSS, and corroborate
this with extensive numerical experiments on a 
variety of model-based and adversarial synthetic matrices. 
The higher accuracy of the CSS methods produces a
more reliable distinction between identifiable and unidentifiable parameters, as illustrated by their highly consistent performance across across this suite of test matrices.
This is especially critical when the identifiable parameters 
inform subsequent investigations 
\cite{BVV99,Cintron2009, Pearce21}.

\subsection{Contributions}
We advocate a numerically reliable and accurate approach for practical
parameter identifiability analysis: Applying column subset selection to the sensitivity matrix, instead of computing an eigenvalue decomposition of the Fischer information matrix. 

\begin{enumerate}
\item We interpret  algorithms based on eigenvalue decompositions of the Fischer matrix \cite{Joll72} as known
column subset selection (CSS) methods
applied to the sensitivity  matrix (section~\ref{s_back}). This connection allows us to derive rigorous guarantees for the accuracy and reliability of the parameter identification that were previously lacking.
\item Identifiability analysis via CSS 
(section~\ref{sec:IdCSS})
has five advantages:
\begin{enumerate}
\item It broadens the applicability of parameter identifiability analysis by permitting the use of synthetic data generated from an additive observation model.  This is crucial when experimental data are not available or optimization for determining nominal parameter values is not feasible.
\item It incorporates parameter correlation.
\item  It quantifies reliability of the subsets of parameters selected as identifiable and unidentifiable.
\item It establishes criteria for comparing the accuracy
of different algorithms.
\item The implementations are numerically more accurate and reliable than eigenvalue methods applied to the
Fischer matrix, yet
without an increase in computational cost. 

\end{enumerate}

\item We perform extensive numerical experiments 
(section~\ref{sec:physmods})
on sensitivity  matrices from six physical models
(section \ref{sec:RealSensDesc}, Appendix~\ref{sec:supp})
to illustrate the accuracy and reliability of the
CSS methods.
\item Among the four CSS methods (Algorithms \ref{alg_PCAB1}--\ref{alg_srrqr}),
we recommend an implementation based on the strong rank-revealing QR algorithm 
(Algorithm~\ref{alg_srrqr}) because of its rigorous accuracy guarantees for both, identifiable and 
unidentifiable parameters, through bounds that 
have only a polynomial dependence on the number of relevant
parameters, rather than an exponential dependence as
in Algorithms \ref{alg_PCAB1}--\ref{alg_PCAB3}.
 \item We construct an adversarial matrix,
the SHIPS matrix (section~\ref{sec:SynSensDesc}) to amplify accuracy differences among the CSS methods.
Although synthetic, the adversarial matrices
(section~\ref{sec:SynSensDesc})
still admit an interpretation as sensitivity matrices
for certain dynamical systems 
(Appendix~\ref{sec:supp3}).
\end{enumerate}

The CSS algorithms (section~\ref{sec:IdCSS})
are based on existing work and presented with a view towards understanding rather than efficiency. 
In the same vein,
the correctness proofs (section~\ref{s_addr}) are
geared towards exposition: self-contained, as simple as possible, and more general with slightly fewer assumptions. 
With a view towards reproducibility, our implementations are available on \url{https://github.com/kjpearce/CSS-Algs-for-Sens-Identifiability}.

 % Introduction
\section{Parameter sensitivity and identifiability}\label{s_paramid}
We define the notion of parameter identifiability
(section~\ref{s_pi}), and 
present real applications that require it (section~\ref{sec:RealSensDesc}).

\subsection{Parameter identifiability}\label{s_pi}
We assume that a model's quantity of interest $y$, such as a state variable in a system of differential equations, can be expressed as a scalar-valued function of system inputs and parameters,
$y = h(\mathbf{u};\mathbf{q})$.
Here the vector $\vu$ represents system inputs, such as time,  and 
the vector $\vq \in \real^p$ the model parameters. 

We denote the \textit{sensitivity} of $y$ with respect to the parameter $q_j$, evaluated at the $i$th observation and a specific point $\mathbf{q}^*$ in the admissible parameter space, by
\begin{align*}
    s_{ij} = \frac{\partial h_i(\mathbf{u};\mathbf{q})}{\partial q_j}\Bigr|_{\mathbf{q}=\mathbf{q}^*}, \qquad 1\leq i\leq n, \quad 
    1\leq j\leq p.
\end{align*}
The sensitivity matrix is 
$\mchi=(s_{ij})\in\real^{n\times p}$, 
and has more rows than columns, $n \geq p$. 
The parameters $\vq$ are \textit{sensitivity-identifiable} at $\vq^*$ if $\mchi^T\mchi$ is invertible \cite{Cintron2009,Luen1969,Reid1977}. 
Our goal is to determine those columns of $\mchi$ that correspond to the most sensitivity-identifiable and the least sensitivity-identifiable parameters.

\subsection{Practical applications with sensitivity matrices} 
\label{sec:RealSensDesc}
We describe an epidemiological compartment model
in detail (section~\ref{s_rsd1}), and summarize five other 
mathematical models 
together with their
quantities of interest  (section~\ref{s_rsd2}).

\subsubsection{SVIR Model}\label{s_rsd1}
The epidemiological SVIR compartment model
in Figure~\ref{fig:SIRdiags}(c),
models the spread of disease among susceptible $S$, vaccinated $V$, 
infectious $I$, and recovered $R$ in a population of $N$ individuals;
and consists of a coupled system of four ordinary differential
equations with specified initial conditions,
\begin{align*}
    \frac{dS}{dt} &= - \beta \frac{{I}{S}}{N}, &S(0) = S_0 \\
    \frac{dV}{dt} &= \nu S - \alpha \beta \frac{IV}{N}, &V(0) = V_0, \\
    \frac{d{I}}{dt} &= \beta \frac{{I}{S}}{N} + \alpha \beta \frac{IV}{N}  - \gamma {I}, &I(0)=I_0, \\
    \frac{dR}{dt} &= \gamma {I}, &R(0) = R_0.
\end{align*}
The epidemiological parameters
$\vq=\begin{bmatrix}\beta &\nu &\alpha&\gamma\end{bmatrix}^\top$ govern 
the system dynamics; the system input $\vu$ is time $t$; and the quantity of interest is $y = h(t; \vq) \equiv I(t; \vq)$
the number of infectious individuals at time $t$.
Discretization with respect to time $t = t_i$, $1\leq i\leq n$,
produces a sensitivity matrix evaluated at 
a nominal  point $\vq^*$,
\begin{align*}
    \mchi = \begin{bmatrix} \frac{\partial h(t_i; \vq)}{\partial \beta} & \frac{\partial h(t_i; \vq)}{\partial \nu} & \frac{\partial h(t_i; \vq)}{\partial \alpha} & \frac{\partial h(t_i; \vq)}{\partial \gamma} \end{bmatrix}\Bigr|_{\mathbf{q}=\mathbf{q}^*} \in\real^{n\times 4}.
\end{align*}
Nominal parameter values are often selected
from the literature, 
as shown in Table~\ref{tab:epiTab}, or
as solutions of inverse problems with available data.
Numerical sensitivities in $\mchi$ are estimated from derivative approximations, such as finite difference or complex-step approximations \cite{Lyn67, LynMo67}. 

\subsubsection{Six models from physical applications}\label{s_rsd2}
We present numerical experiments (section~\ref{s_appl})
for the six models below, with quantities 
of interest in Table~\ref{tab:QOI}.
More details can be found in section~\ref{sec:supp}.
\begin{itemize}
    \item SVIR: See above.
    \item SEVIR \cite{Per20}: This extension
    of SVIR model adds an additional compartment for individuals $E$ who have been exposed but are not yet infectious.
    \item COVID \cite{Perk20}: This extension
    of SEVIR splits the infectious group into compartments for asymptomatic, symptomatic, and hospitalized individuals.
    \item HGO \cite{Hai}: This model for the biomechanical deformation of the left pulmonary artery vessel wall is based
    on nonlinear hyperelastic  structural relations, and calibrated to in vitro experiments on normal and hypertensive mice.
    \item Wound \cite{Pearce21}: This model for in vitro fibrin matrix polymerization during hemostasis
    concerns clot formation 
    during the first stage of wound healing, 
    and is based on biochemical reaction kinetics. 
    \item Neuro \cite{hart2019}: This model of the neurovascular coupling (NVC) response
    describes local changes in vascular resistance that result from neuronal activity, and is based on nonlinear ODEs.
\end{itemize}

\begin{table}[H]
\label{tab:QOI}
    \centering
    \begin{tabular}{c|c|c|c}
     Model  &  Type  & $p$ & Quantity of Interest    \\
     \hline
    SVIR & Epidemiological & 4 & \# Infectious individuals \\
    \hline
    SEVIR & Epidemiological & 5 & \# Infectious individuals \\
    \hline 
    COVID & Epidemiological & 8 & \# Infectious (sympt., asymp., hospitalized) \\
    \hline 
    HGO & Cardiovascular & 8 & Vessel lumen area and wall thickness  \\
    \hline
    Wound & Wound Healing & 11 & Fibrin matrix (\textit{in vitro}  clot) concentration \\
    \hline
    Neuro & Neurological & 175 & Blood  oxyhemoglobin concentration \\
    \end{tabular}
    \caption{Number of parameters $p$ and quantities of interest for the 
    models in section~\ref{sec:RealSensDesc}}
\end{table}

 % Background
 \section{Background}\label{s_back}
We express sensitivity
analysis on the eigenvectors of the Fischer matrix
$\mf=\ms^T\ms$  
as column subset selection on the sensitivity matrix $\ms$. 

After briefly introducing notation (section~\ref{s_not}),
we review identfiability analysis based on eigenvectors of the 
Fischer matrix (section~\ref{s_fev}), the singular value decomposition
of the sensitivity matrix (section~\ref{s_svd}),
column subset selection on the sensitivity matrix (section \ref{subsec:CSS}),
determination of the number $k$ of identifiable parameters
(section~\ref{s_k}), and finally the implementation
of column subset selection via
QR decompositions (section~\ref{s_ccqr}).

\subsection{Notation}\label{s_not}
We denote matrices by bold upper case letters. 
The identity matrix is
\begin{align*}
\mi_p\equiv\begin{bmatrix} 1&& \\ &\ddots& \\  &&1 \end{bmatrix}=
\begin{bmatrix} \ve_1 & \cdots & \ve_p\end{bmatrix}\in\real^{p\times p}
\end{align*}
with columns that are the canonical vectors $\ve_j\in\real^{p}$.

We assume that 
the sensitivity matrix $\mchi\in\real^{n \times p}$ 
is tall and skinny, with at least
as many rows as columns, $n\geq p$.
The $p$ columns of $\mchi$ represent parameters
and its rows represent observations. 
The Fischer information matrix is the cross product 
matrix $\mf\equiv \mchi^T\mchi\in\real^{p\times p}$, where
the superscript $T$ denotes the transpose.

 \subsection{Eigenvalue decomposition of the Fischer matrix}\label{s_fev}
 Existing methods \cite{Joll72,Miao2011} select parameters
by inspecting the eigenvectors of the Fischer matrix
$\mf=\mchi^T\mchi\in\real^{p\times p}$.
Since it is real symmetric positive semi-definite, its eigenvalue decomposition has the form
\begin{align}\label{e_fev}
\mf=\mv \begin{bmatrix}\lambda_1&&\\ &\ddots&\\ && \lambda_p\end{bmatrix}\mv^T,
\qquad \lambda_1\geq\cdots \geq \lambda_p\geq 0,
\end{align}
where $\lambda_j$ are the eigenvalues.
The eigenvector matrix $\mv\in\real^{p\times p}$ is an orthogonal matrix
with $\mv^T\mv=\mi_p=\mv\mv^T$.
Its columns and elements are
\begin{align*}
\mv=\begin{bmatrix}\vv_1& \cdots & \vv_p\end{bmatrix}=
\begin{bmatrix}v_{11} & \cdots & v_{1p}\\ \vdots & & \vdots \\ 
v_{p1} & \cdots & v_{pp}\end{bmatrix}
\end{align*}
In particular,  the trailing column
$\vv_p$ is an eigenvector associated
with a smallest eigenvalue $\lambda_p$, so $\mf\vv_p =\lambda_p\vv_p$. If $\lambda_p>0$, then  $\mf$ is nonsingular.

The parameter with index~$j$ is represented by column $j$ of $\ms$.
The corresponding column of the Fischer matrix is
\begin{align*}
\ms^T\ms\ve_j=
\mf\ve_j=
\mv \begin{bmatrix}\lambda_1&&\\ &\ddots&\\ && \lambda_p\end{bmatrix}
\mv^T\ve_j \quad
 \text{where}\quad
\mv^T\ve_j =\begin{bmatrix}v_{j1}\\ \vdots \\ v_{jp}\end{bmatrix},
\qquad 1\leq j\leq p.
\end{align*}
Thus, column $j$ of $\mchi$ depends on column $j$ of $\mv^T$
which, in turn, contains element $j$ of each eigenvector.

\smallskip

\begin{center}
\shadowbox{\shortstack{
Selecting element~$j$ of any eigenvector of $\mf=\ms^T\ms$\\
amounts to
selecting the parameter with index~$j$ in $\ms$.}}
\end{center}\smallskip

\begin{caution}
Explicit formation of the Fischer matrix $\mf=\ms^T\ms$ can lead to
 significant loss of information, thus affecting
 subsequent practical identifiability analysis.  
 
For instance \cite[Section 5.3.2]{GovL13}, in customary double 
precision floating point arithmetic with unit roundoff 
$2^{-53}\approx 1.1\cdot 10^{-16}$, 
the sensitivity matrix
\begin{align*}
\mchi=\begin{bmatrix} 1&1\\ 10^{-9} &0\\ 0& 10^{-9}\end{bmatrix}
\end{align*}
has linearly independent columns, and $\rank(\mchi)=2$. In contrast,
the Fischer information matrix computed in double precision 
floating point arithmetic
\begin{align*}
\fl(\mchi^T\mchi)=\begin{bmatrix} 1& 1\\ 1& 1\end{bmatrix}
\end{align*}
is singular, because  the diagonal elements computed in double precision are
\begin{align*}
\fl(1+10^{-9}\cdot 10^{-9})=\fl(1+10^{-18})=1,
\end{align*}
where the operator $\fl(\cdot)$ represents the
output of a computation in floating point arithmetic.
\end{caution}

\subsection{Singular value decomposition of the sensitivity matrix}\label{s_svd}
We avoid the explicit formation of the Fischer matrix
$\mf=\mchi^T\mchi$, and instead
operate directly on the sensitivity matrix $\mchi$,
without increasing the computation time.

This is done  with the help of the (thin) singular value 
decomposition (SVD) 
\cite[section 8.6]{GovL13}
\begin{align}\label{e_SVDS}
\mchi=\mU \begin{bmatrix}\sigma_1&&\\ &\ddots&\\ && \sigma_p \end{bmatrix}\mv^T,
\qquad \sigma_1\geq\cdots \geq \sigma_p\geq 0,
\end{align}
where $\sigma_j$ are the singular values of $\mchi$,
the left singular vector matrix $\mU\in\real^{n\times p}$ has orthonormal
columns with $\mU^T\mU=\mi_p$, and
 the right singular vector matrix $\mv$is 
is identical
to the orthogonal matrix in (\ref{e_fev}).

Substituting the SVD of $\mchi$ into $\mf$ 
gives (\ref{e_fev}) with 
eigenvalues $\lambda_j=\sigma_j^2$, $1\leq j\leq p$.
Thus, squared singular values of~$\mchi$ are the eigenvalues of 
$\mf$, and the right singular vectors of $\mchi$ are eigenvectors of 
$\mf$.

\smallskip

\begin{center}
\shadowbox{\shortstack{
Selecting element~$j$ of any right  singular vector of $\mchi$\\
amounts to
selecting the parameter\\ with index~$j$
in column $j$ of $\ms$.}}
\end{center}
\smallskip

As a consequence,
all  information provided by the eigenvalue decomposition 
of the Fischer matrix $\mf=\mchi^T\mchi$ is available from the SVD of 
the sensitivity matrix~$\mchi$.
Computation of the SVD is not more expensive; see Remark~\ref{r_CH}.

\subsection{Column subset selection on the sensitivity matrix}
\label{subsec:CSS}
We go a step further, and select the parameters
directly from the sensitivity matrix $\mchi$, rather 
than detouring through an eigenvalue or singular value decomposition.

Specifically, we compute a permutation matrix $\mP\in\real^{p\times p}$ 
that reorders the columns of the sensitivity matrix $\mchi$,
\begin{align}\label{e_spart}
 \mchi\mP =\begin{bmatrix} \mchi_1 & \mchi_2\end{bmatrix} 
\end{align}
so that the, say $k$ columns of $\mchi_1$ represent the identifiable
parameters, and the $p-k$ columns of $\mchi_2$ the unidentifiable
parameters.  

In practice, one wants the columns of $\mchi_1$ to represent an approximate basis for $\range(\mchi)$.
A basis satisfies two criteria: Its vectors are linearly
independent, and they span the host space. 
\begin{enumerate}
\item Linear independence of the columns of $\mchi_1\in\real^{n\times k}$ is quantified by the magnitude of its smallest singular
value, which is
bounded above by the $k$th largest singular value of the host matrix,
\begin{align}\label{e_c1}
\sigma_k(\mchi_1)\leq \sigma_k.
\end{align}
The larger $\sigma_k(\mchi_1)$, the more
linearly independent the columns of $\mchi_1$.
A more specific statement is presented in (\ref{e_c1b}).
\item Spanning the host space $\range(\mchi)$
is quantified by the accuracy of $\mchi_1$ as a low-rank
approximation of the host 
matrix~$\mchi$. One measure of accuracy
is the residual norm, which is bounded below by the $(k+1)$st
singular value of the host 
matrix\footnote{The superscript $\mchi_1^{\dagger}$ denotes the Moore-Penrose inverse, and the equalities follow from the Moore-Penrose property $\ms_1\ms_1^{\dagger}\ms_1=\ms_1$ and the unitary invariance of
the two-norm with regard to the permutation $\mP$.}
\begin{align}\label{e_c2}
\|(\mi-\mchi_1\mchi_1^{\dagger})\mchi\|_2=
\|(\mi-\mchi_1\mchi_1^{\dagger})\mchi_2\|_2\geq \sigma_{k+1}.
\end{align}
The smaller the residual, the better $\range(\mchi_1)$
spans the host space.
Criterion~(\ref{e_c2}) is a special case of the subsequent (\ref{e_c2b}).
\end{enumerate}

\smallskip

\begin{center}
\shadowbox{\shortstack{
Identifiable parameters
are the `most linearly independent' columns' of $\mchi$.\\
Unidentifiable parameters are the 
`most linearly dependent' columns of $\mchi$.}}
\end{center}
\smallskip

Algorithms \ref{alg_PCAB1}
and~\ref{alg_PCAB3} select unidentifiable parameters,
Algorithm~\ref{alg_PCAB4} selects identifiable parameters, while Algorithm~\ref{alg_srrqr} selects both.

\begin{caution} The separation into linearly dependent and independent columns
is highly non-unique. For instance, the matrix 
\begin{align*}
\mchi=\begin{bmatrix} 1 & 0& 1&0\\ 0&1&0&1\\
0&0&0&0\\ 0&0&0&0\end{bmatrix}
\end{align*}
has $\rank(\mchi)=k=2$ with $\sigma_1=\sigma_2=\sigma_k=\sqrt{2}$ and 
$\sigma_4=\sigma_3=\sigma_{k+1}=0$.
Moving two linearly independent columns of $\mchi$ 
to the front can be accomplished
by any of the following permutation matrices $\mP$,
\begin{align*}
\begin{bmatrix} 1&0&0&0\\ 0&1&0&0\\ 0&0&1&0\\ 0&0&0&1\end{bmatrix},\quad
\begin{bmatrix} 0&0&1&0\\ 0&0&0&1\\ 1&0&0&0\\ 0&1&0&0\end{bmatrix},\quad
\begin{bmatrix} 1&0&0&0\\ 0&0&1&0\\ 0&0&0&1\\ 0&1&0&0\end{bmatrix},\quad
\begin{bmatrix} 0&0&1&0\\ 1&0&0&0\\ 0&1&0&0\\ 0&0&0&1\end{bmatrix}
\end{align*}
to produce the same matrix
\begin{align*}
\mchi_1=\begin{bmatrix}1&0\\ 0&1\\0&0\\0&0\end{bmatrix}
\end{align*}
with residual (\ref{e_c2}) equal to $\sigma_3=0$.
\end{caution}

One can require the criteria
(\ref{e_c1}) or (\ref{e_c2}) to hold either for all $1\leq k\leq p$
\cite{HCGM21},
or else for only one specific $k$ \cite{ChI91a,GuE96}.
In the latter case, section~\ref{s_k} discusses approaches for selecting~$k$. 

\subsection{Choosing the number k of identifiable parameters}\label{s_k}
If one knows a bound $\eta$ on the error or noise
in the elements of~$\mchi$, 
one can use criterion (\ref{e_c2})
to designate as small all those singular values below~$\eta$, 
in the absolute or the relative sense,
\begin{align*}
\sigma_{k+1}\leq \eta \qquad \text{or}\qquad 
\sigma_{k+1}\leq \eta\,\sigma_1.
\end{align*}
For instance, if $\eta$ bounds the relative error in the elements of $\mchi$, then the value
of~$k$ determined by $\sigma_{k+1}\leq \eta\,\sigma_1$ is called the
\textit{numerical rank} of $\mchi$ \cite[Definition 2.1]{GoKS76},
\cite[section 5.4.2]{GovL13}.
If $\mchi$ is accurate to double precision unit roundoff, 
then $\eta \approx 1.1\cdot 10^{-16}$.

Alternatively, one can use criterion (\ref{e_c1}) to designate as large 
all those singular values exceeding~$\eta$, in the absolute or 
the relative sense,
\begin{align*}
\sigma_k> \eta \qquad \text{or}\qquad \sigma_k >\eta\,\sigma_1.
\end{align*}

If the accuracy of the elements in $\mchi$ is unknown, but its 
singular values contain a prominent gap, then one can choose $k$ to capture this gap,
\begin{align*}
\sigma_1\geq \cdots \geq \sigma_k \ \gg\ \sigma_{k+1}\geq \cdots \geq \sigma_p
\end{align*}
An upgrade \cite[Algorithm 5]{GuE96} of Algorithm~\ref{alg_srrqr} 
looks for
a large gap between adjacent singular values, 
in order to
compute $k$ automatically  \cite[Remark 1]{GuE96}.
\smallskip

\begin{center}
\shadowbox{\shortstack{
The number $k$ of identifiable parameters\\
can be chosen as the numerical rank of $\mchi$,\\
or based on a large gap in the singular values.}}
\end{center}
\smallskip

\subsection{Implementing column subset selection with QR decompositions}\label{s_ccqr}
We show how to compute, by means of pivoted QR decompositions,
permutation matrices $\mP$ that try
to optimize criteria (\ref{e_c1})
or (\ref{e_c2}).  As a matter of exposition, we introduce
plain QR decompositions  (section~\ref{s_qrplain}), 
pivoted QR decompositions (section~\ref{s_qrpivot}) and then rank revealing QR decompositions (section~\ref{s_qrreveal}).

\subsubsection{QR decompositions}\label{s_qrplain}
Assume that the sensitivity matrix $\mchi\in\real^{n\times p}$
has full column rank with $\rank(\mchi)=p$.
A `thin QR decomposition' \cite[section 5.2]{GovL13}, \cite[Chapter 19]{Higham2002}
is  a basis transformation that  transforms
the basis for $\range(\mchi)$ from linearly independent 
columns of $\mchi$ to orthonormal columns of $\mq$,
\begin{align}\label{e_qr}
\mchi=\mq\mr.
\end{align}
Here $\mq\in\real^{n\times p}$ has orthonormal columns with $\mq^T\mq=\mi_p$,
and the nonsingular upper triangular matrix $\mr\in\real^{p\times p}$ represents
an easy relation between the two bases.
Substituting (\ref{e_qr}) into $\mchi$ gives for  Fischer information matrix
\begin{align*}
\mf=\mchi^T\mchi=\mr^T\mr.
\end{align*}
Thus the eigenvalues of $\mf$ are equal to the squared singular values of
the triangular matrix~$\mr$.

\subsubsection{Pivoted QR decompositions}\label{s_qrpivot}
These decompositions have more flexibility because they can additionally 
permute (pivot)
the columns of $\mchi$ to compute an orthonormal basis for $\range(\mchi)$ \cite[5.4.2]{GovL13}, \cite[Chapter 19]{Higham2002},
\begin{align}\label{e_pqr}
\mchi\mP = \mq \mr,
\end{align}
where $\mP\in\real^{p\times p}$ is the permutation matrix in (\ref{e_spart});
$\mq\in\real^{n\times p}$ has orthonormal columns with $\mq^T\mq=\mi_p$;
and $\mr\in\real^{p\times p}$ is upper triangular.  
Substituting the factorization (\ref{e_pqr}) into 
the sensitivity matrix $\mchi$ gives for  Fischer matrix
\begin{align*}
\mf=\mchi^T\mchi=\mP\,\mr^T\mr\,\mP^T.
\end{align*}
Since permutation matrices are orthogonal matrices, 
the eigenvalues of $\mf$ are still equal to the squared singular values of $\mr$, while each eigenvector of $\mr^T\mr$
is a  permutation of the corresponding eigenvector of $\mf$.

Algorithms \ref{alg_PCAB1}--\ref{alg_srrqr} start with
a preliminary QR decomposition 
to reduce the dimension of the matrix. The
following remark shows that such a 
preliminary decomposition is also effective prior
to an SVD computation, and the proofs in 
Appendix~\ref{s_addr} exploit this.

\begin{remark}\label{r_CH}
A preliminary QR decomposition $\mchi\mP=\mq\mr$ is an efficient way
to compute the SVD of a 
dense matrix $\mchi\in\real^{n\times p}$ with $n\geq p$  \cite{ChanHansen}, since it
reduces the dimension for the SVD from that of a tall and skinny
$n\times p$ matrix down to that of a small square $p\times p$ matrix
with the same dimension as the Fischer matrix $\mf=\ms^T\ms$.

To see this, compute the pivoted QR factorization 
$\mchi\mP=\mq\mr$, and 
let the upper triangular $\mr$ have an SVD
\begin{align*}
\mr=\mU_r\begin{bmatrix}\sigma_1&&\\ &\ddots &\\&& \sigma_p\end{bmatrix}
\mv^T,
\end{align*}
where $\mU_r$ and $\mv\in\real^{p\times p}$ are orthogonal matrices.
Then the SVD of the permuted sensitivity matrix $\mchi\mP$ is 
\begin{align*}
\mchi\mP=(\mq \mU_r)
\begin{bmatrix}\sigma_1&&\\ &\ddots &\\&& \sigma_p\end{bmatrix} \mv^T,
\end{align*}
where the left singular vector matrix $\mq \mU_r\in\real^{n\times p}$ 
has orthonormal columns. 

This approach retains the asymptotic complexity of an SVD of $\mchi$,
but has the advantage of reducing the actual operation count,
and, in particular, reducing the problem dimension to
that of the Fischer matrix $\mf=\ms^T\ms$.
\end{remark}

\subsubsection{Rank revealing QR decompositions}\label{s_qrreveal}
\label{s_rrqr}
These pivoted QR decompositions are designed to `reveal' the numerical rank of 
a matrix $\mchi$ that is rank deficient, or ill-conditioned with regard to left inversion
\cite[section 2]{ChI91a}, \cite[5.4.2]{GovL13}. \cite[section 1.1]{GuE96}.
Although there are numerous  ways to compute such decompositions
 \cite{ChI91a,GuE96}, most share the same overall strategy.
 
Assume the sensitivity matrix has numerical $\rank(\mchi)\approx k$,
where $1\leq k<p$. Partition 
the pivoted QR decomposition (\ref{e_pqr}) commensurately with the column partitioning~(\ref{e_spart}),
\begin{align}\label{e_part}
 \underbrace{\begin{bmatrix} \mchi_1 & \mchi_2\end{bmatrix}}_{\mchi\mP} =
 \mchi
 \underbrace{\begin{bmatrix} \mP_1 & \mP_2\end{bmatrix}}_{\mP}=
 \underbrace{\begin{bmatrix} \mq_1 & \mq_2\end{bmatrix}}_{\mq}
\underbrace{\begin{bmatrix} \mr_{11} &  \mr_{12}\\ \vzero & \mr_{22}
\end{bmatrix}}_{\mr},
%=\begin{bmatrix}\mq_1\mr_{11}& \mq_1\mr_{12}+\mq_2\mr_{22} \end{bmatrix}
\end{align}
with submatrices 
$\mP_1\in\real^{p\times k}$, $\mq_1\in\real^{n\times k}$, 
$\mr_{11}\in\real^{k\times k}$, and $\mr_{22}\in\real^{(p-k)\times (p-k)}$.

Since $\mchi_1=\mq_1\mr_{11}$,
the leading diagonal block $\mr_{11}$ has the same singular values as the matrix $\mchi_1$ of
identifiable parameters
\begin{align}\label{e_c1b}
\sigma_j(\mchi_1)=\sigma_j(\mr_{11}),\qquad 1\leq j\leq k.
\end{align}
Similarly, since
\begin{align}\label{e_c2b}
\sigma_j((\mi-\mchi_1\mchi_1^{\dagger})\mchi)=
\sigma_j((\mi-\mchi_1\mchi_1^{\dagger})\mchi_2)=\sigma_j(\mr_{22}),\qquad 
1\leq j\leq p-k,
\end{align}
the trailing diagonal block $\mr_{22}$ has the same non-zero
singular values as the residuals of the low-rank approximation
of $\range(\mchi)$ by $\range(\mchi_1)$.

We call a QR decomposition (\ref{e_part})  qualitatively `rank-revealing'
if it tries to optimize subselection criteria (\ref{e_c1}) or (\ref{e_c2}), 
that is,
\begin{align*}
\sigma_k(\mr_{11})\approx\sigma_k \qquad \text{or}\qquad 
\sigma_1(\mr_{22})=\|\mr_{22}\|_2\approx\sigma_{k+1}.
\end{align*}
The first criterion tries to produce a well conditioned basis $\ms_1=\mq_1\mr_{11}$, and its
approximation $\ms_1\mP_1^T\approx \ms$.
The second criterion aligns with the popular and robust 
requirement
$\|(\mi-\mchi_1\mchi_1^{\dagger})\mchi\|_2\approx\sigma_{k+1}$
for low-rank approximations \cite{DI18}.
\smallskip

\begin{center}
\shadowbox{\shortstack{
Rank-revealing QR decompositions try to select
as identifiable
parameters\\
those columns of $\mchi$ that
are the most linearly independent or\\
that approximate well the unidentifiable parameters.}}
\end{center}
\smallskip

Rigorous, stringent versions of the subselection criteria (\ref{e_c1}) and (\ref{e_c2}) are presented in 
\cite[Section 1.2]{GuE96} and Theorem~\ref{t_srrqr}. % identifiability as column subset selection 
\section{Identifiability as column subset selection}
\label{sec:IdCSS}
We express practical identifiability analysis  
\cite[Definition 5.11]{Miao2011}, \cite[page 4 of 21]{Quaiser}
as column subset selection, to quantify accuracy
 and to compare the accuracy of different algorithms.
We start with Jollife's methods \cite{Joll72,Miao2011}: PCA method B1 (section~\ref{s_PCAB1}), 
PCA method B4 (section~\ref{s_PCAB4}), and PCA method B3 (section~\ref{s_PCAB3}),
and then propose the strong rank-revealing QR factorization \cite{GuE96} as 
the most accurate option for practical 
identifiability analysis (section~\ref{s_srrqr}). 

Algorithms \ref{alg_PCAB1}--\ref{alg_srrqr} input a tall and
skinny sensitivity 
matrix~$\mchi$ and the number~$k$ of identifiable parameters, say from section~\ref{s_k}; and output the factors of a pivoted QR
decomposition $\ms\mP=\mq\mr$.
\smallskip

\begin{center}
\shadowbox{\shortstack{
We recommend Algorithm~\ref{alg_srrqr} in section~\ref{s_srrqr}.\\
It has the most rigorous and realistic accuracy guarantees\\
for both, identifiable and unidentifiable parameters.}}
\end{center}
\smallskip

The algorithms are formulated with a focus on understanding, rather than efficiency.

\subsection{PCA method B1}\label{s_PCAB1}
This method  \cite[section 2.2]{Joll72}, \cite[(5.13)]{Miao2011}
selects un-identifiable parameters, by detecting large-magnitude 
components in the eigenvectors 
$\vv_{k+1},\ldots, \vv_p$
corresponding to the $p-k$ smallest eigenvalues of the
Fischer matrix $\mf$, starting from the smallest eigenvalue.

Method B1 starts with a unit-norm eigenvector $\vv_p$ corresponding to $\lambda_p$,
picks a magnitude largest element in $\vv_p$,
\begin{align*}
|v_{{m_1},p}|=\max_{1\leq j\leq p}{|v_{jp}|},
\end{align*}
and designates the parameter with index $m_1$ as unidentifiable.
Method B1 repeats this on eigenvectors corresponding to eigenvalues 
$\lambda_{p-1}\leq \cdots\leq \lambda_{k+1}$ in that order, by selecting magnitude-largest
elements that have not been selected previously,
\begin{align*}
|v_{{m_\ell},\ell}|=
\max_{\stackrel{1\leq j\leq p}{j\neq m_1, \ldots, m_{p-\ell+1}}}{|v_{j\ell}|},
\qquad \ell =p-1, \ldots, k+1,
\end{align*}
and declares the parameters with indices $m_1, \ldots m_{p-k}$
as unidentifiable.

\subsubsection*{Expressing PCA method B1 as column subset selection}
PCA method~B1 is almost identical to the subset selection algorithm
in \cite[Section 3]{Chan1987}, which is also 
\cite[Algorithm Chan-II]{ChI91a},
and is related to the algorithms in \cite{Foster1986,GraggStewart1976}.

Algorithm~\ref{alg_PCAB1}, which represents \cite[Algorithm RRQR(r)]{Chan1987},
selects $p-k$ unidentifiable parameters $\mchi_2$ to
optimize subset selection criterion (\ref{e_c2}) and moves them to the back
of the matrix.
Once a column for $\ms_2$ has been identified, Algorithm~\ref{alg_PCAB1}
ignores it from then on, and continues on a lower-dimensional submatrix. 

\begin{algorithm}\caption{Column subset selection version of PCA B1}\label{alg_PCAB1}
\begin{algorithmic}
\REQUIRE $\mchi\in\real^{n\times p}$ with $n\geq p$,  $1 \leq k < p$ 
\STATE Set $\mP=\mi_p$ 
\STATE Compute decomposition (\ref{e_pqr}): $\mchi\mP=\mq\mr\qquad$ 
\COMMENT{Unpivoted QR of $\mchi$}
\FOR{$\ell=p:k+1$}
\STATE $\qquad$ \COMMENT{If $\ell=p$, then $\mr_{11}=\mr$ }
\STATE Partition 
$\mr=\begin{bmatrix} \mr_{11} & \mr_{12} \\ \vzero& \mr_{22}\end{bmatrix}$
where $\mr_{11}\in\real^{\ell\times \ell}\quad$ 
\COMMENT{Focus on leading $\ell\times \ell$ block}
\STATE Compute right singular vector $\vv\in\real^{\ell}$ of $\mr_{11}$ 
corresponding to 
$\sigma_{\ell}(\mr_{11})$
\STATE Compute permutation $\widetilde{\mP}\in\real^{\ell\times\ell}$ so that
$|(\widetilde{\mP}^T\vv)_{\ell}|=\|\vv\|_{\infty}$ \\ 
$\qquad$ \COMMENT{Move magnitude-largest element of $\vv$ to bottom}
\STATE Compute QR decomposition  (\ref{e_qr}): 
$\mr_{11}\widetilde{\mP}=\widetilde{\mq}\widetilde{\mr}_{11}\qquad$
\COMMENT{Unpivoted QR of $\mr_{11}\widetilde{\mP}$}
\STATE Update 
$\mq:=\mq\begin{bmatrix}\widetilde{\mq}&\vzero\\ \vzero & \mi_{p-\ell} \end{bmatrix}$, 
$\mP:=\mP\begin{bmatrix}\widetilde{\mP}&\vzero\\ \vzero&\mi_{p-\ell}\end{bmatrix}$,
$\mr:=\begin{bmatrix}\widetilde{\mr}_{11} & \widetilde{\mq}^T\mr_{12}\\ \vzero& \mr_{22}
\end{bmatrix}$
\ENDFOR
\RETURN $\mP$, $\mq$, $\mr$
\end{algorithmic}
\end{algorithm}

Theorem~\ref{t_PCAB1} shows that the unidentifiable parameters 
$\mchi_2$ from Algorithm~\ref{alg_PCAB1} can be 
interpreted as column subsets satisfying criterion~(\ref{e_c2}).

\begin{theorem}\label{t_PCAB1}
Let $\mchi\in\real^{n\times p}$ with $n\geq  p$ be the sensitivity matrix,
and $1\leq k<p$. Algorithm~\ref{alg_PCAB1} computes a pivoted 
QR decomposition 
\begin{align*}
\mchi\mP=\begin{bmatrix}\mchi_1 & \mchi_2\end{bmatrix}=
\begin{bmatrix}\mq_1 & \mq_2\end{bmatrix}
\begin{bmatrix}\mr_{11}&\mr_{12}\\\vzero & \mr_{22}\end{bmatrix},
\qquad \mr_{22}\in\real^{(p-k)\times (p-k)},
\end{align*}
where 
\begin{align*}
\sigma_{k+1}\leq \|(\mi-\mchi_1\mchi_1^{\dagger})\mchi_2\|_2=
\|\mr_{22}\|_2&\leq 2^{p-k-1}\sigma_{k+1}.
\end{align*}
If the mumerical $\rank(\mchi)= k$, then the columns of $\mchi_2$
represent the $p-k$ unidentifiable parameters.
\end{theorem}

\begin{proof}
The  equality follows from (\ref{e_c2b}), while the lower bound 
follows from interlacing~(\ref{e_inter}).
The upper bound is derived in section~\ref{s_PCAB1p}, and in particular in 
Lemma~\ref{l_PCAB1tp}.
\end{proof}

Theorem~\ref{t_PCAB1} bounds the residual in the low rank
approximation $\mchi_1$ according to criterion~(\ref{e_c2}).
Like many subset selection bounds, the upper bound can 
be achieved by artificially contrived matrices 
\cite[section 8.3]{Higham2002}, but tends to be quantitatively 
pessimistic in practice.
Fortunately, it is informative from a qualitative perspective.

\subsection{PCA method B4}\label{s_PCAB4}
This method  \cite[section 2.2]{Joll72}, \cite[(5.15)]{Miao2011},
\cite[Appendix C]{Quaiser}
selects identifiable parameters, by detecting large-magnitude 
components in the eigenvectors 
$\vv_1, \ldots, \vv_k$
corresponding to the $k$ largest eigenvalues of the
Fischer matrix, starting from the largest  eigenvalue.
Our detailed interpretation of the algorithm follows that in
\cite[Appendix C, Third Criterion]{Quaiser}.

Method B4
starts with a unit-norm eigenvector $\vv_1$ corresponding to $\lambda_1$,
picks a magnitude largest element in $\vv_1$,
\begin{align*}
|v_{{m_1},1}|=\max_{1\leq j\leq p}{|v_{j1}|},
\end{align*}
and declares the parameter with index $m_1$ as identifiable.
Method B4 repeats this on eigenvectors corresponding to eigenvalues 
$\lambda_{2}\geq \cdots\geq \lambda_k$ in that order, by selecting magnitude-largest
elements that have not been selected previously,
\begin{align*}
|v_{{m_\ell},\ell}|=
\max_{\stackrel{1\leq j\leq p}{j\neq m_1, \ldots, m_{\ell}}}{|v_{j\ell}|},
\qquad \ell =2, \ldots, k,
\end{align*}
and declares the parameters with indices $m_1, \ldots,m_k$ as identifiable.

\subsubsection*{Expressing PCA method B4 as column subset selection}
PCA method B4 is almost identical to the subset selection algorithm
in \cite[Section 3]{ChanHansen}, which is 
also \cite[Algorithm Chan-I]{ChI91a}.

Algorithm~\ref{alg_PCAB4}, which represents \cite[Algorithm L-RRQR]{ChanHansen}, 
selects $k$ identifiable parameters $\ms_1$ to
optimize subset selection criterion (\ref{e_c1}) and
moves them to the front of the matrix.
Once a column for $\ms_1$ has been identified, Algorithm~\ref{alg_PCAB4}
ignores it from then on, and continues on a lower-dimensional submatrix.

\begin{algorithm}\caption{Column subset selection version of PCA B4}\label{alg_PCAB4}
\begin{algorithmic}
\REQUIRE $\mchi\in\real^{n\times p}$ with $n\geq p$, $1 \leq k < p$ 
\STATE Set $\mP = \mi_p$
\STATE Compute decomposition (\ref{e_pqr}): $\mchi\mP=\mq\mr\qquad$ 
\COMMENT{Unpivoted QR of $\mchi$}
\FOR{$\ell=1:k$} 
\STATE $\qquad$ \COMMENT{If $\ell=1$, then $\mr_{22}=\mr$ }
\STATE Partition 
$\mr=\begin{bmatrix} \mr_{11} & \mr_{12} \\ \vzero& \mr_{22}\end{bmatrix}$
where $\mr_{22}\in\real^{(p-\ell+1)\times (p-\ell+1)}$ \\
$\qquad$ \COMMENT{Focus on trailing $(p-\ell+1)\times (p-\ell+1)$ block}
\STATE Compute right singular vector $\vv\in\real^{p-\ell+1}$ of $\mr_{22}$ 
corresponding to 
$\sigma_{1}(\mr_{22})$
\STATE Compute permutation $\widetilde{\mP}\in\real^{(p-\ell+1)\times(p-\ell+1)}$ so that
$|(\widetilde{\mP}^T\vv)_1|=\|\vv\|_{\infty}$ \\
$\qquad$ \COMMENT{Move magnitude-largest element of $\vv$ to top}
\STATE Compute QR decomposition (\ref{e_qr}): 
$\mr_{22}\widetilde{\mP}=\widetilde{\mq}\widetilde{\mr}_{22}\qquad$
\COMMENT{Unpivoted QR of $\mr_{22}\widetilde{\mP}$}
\STATE Update 
$\mq:=\mq\begin{bmatrix}\mi_{\ell-1}&\vzero\\ \vzero& \widetilde{\mq} \end{bmatrix}$, 
$\mP:=\mP\begin{bmatrix}\mi_{\ell-1}&\vzero\\ \vzero&\widetilde{\mP}\end{bmatrix}$,
$\mr:=\begin{bmatrix}\mr_{11} & \mr_{12}\\ \vzero& \widetilde{\mr}_{22}
\end{bmatrix}$
\ENDFOR
\RETURN $\mP$, $\mq$, $\mr$
\end{algorithmic}
\end{algorithm}

Theorem~\ref{t_PCAB4} shows that the identifiable parameters $\mchi_1$
from Algorithm~\ref{alg_PCAB4} can be 
interpreted as parameters that satisfy criterion~(\ref{e_c1}).

\begin{theorem}\label{t_PCAB4}
Let $\mchi\in\real^{n\times p}$ with $n\geq  p$ be the sensitivity matrix, 
and $1\leq k<p$. Then Algorithm~\ref{alg_PCAB4} computes
a QR decomposition 
\begin{align*}
\mchi\mP=\begin{bmatrix}\mchi_1 & \mchi_2\end{bmatrix}=
\begin{bmatrix}\mq_1 & \mq_2\end{bmatrix}
\begin{bmatrix}\mr_{11}&\mr_{12}\\\vzero & \mr_{22}\end{bmatrix},
\qquad \mr_{11}\in\real^{k\times k},
\end{align*}
where
\begin{align*}
2^{-k+1}\sigma_k\leq \sigma_k(\mr_{11})
= \sigma_k(\mchi_1)\leq \sigma_k.
\end{align*}
If numerical $\rank(\mchi)= k$, then the columns of $\mchi_1$ represent the $k$ identifiable parameters.
\end{theorem}

\begin{proof}
The  equality follows from (\ref{e_c1b}), while 
the upper bound follows from interlacing~(\ref{e_inter}).
The lower bound is derived in section~\ref{s_PCAB4p}, and in particular in 
Lemma~\ref{l_PCAB4tp}.
\end{proof}

Theorem~\ref{t_PCAB4} bounds the linear independence of the columns in 
$\mchi_1$ according to criterion~(\ref{e_c1}). 
As before, the lower bound in Theorem~\ref{t_PCAB4}
can be quantitatively very pessimistic in practice, but 
tends to be qualitatively informative.

\subsection{PCA method B3}\label{s_PCAB3}
This method  \cite[section 2.2]{Joll72}, \cite[(5.14)]{Miao2011}
selects unidentifiable parameters by detecting large squared
row sums in the matrix
$\mv_{k+1:p}\equiv\begin{bmatrix}\vv_{k+1}& \cdots& \vv_p\end{bmatrix}$ of eigenvectors
corresponding to the $p-k$ smallest eigenvalues of the Fischer
matrix.

The squared row norms of $\mv_{k+1:p}$,
\begin{align*}
\omega_j\equiv \left\|\begin{bmatrix}v_{j,k+1} &\cdots & v_{jp} \end{bmatrix}\right\|_2^2=\sum_{\ell=k+1}^p{v_{j\ell}^2},\qquad 1\leq j \leq p,
\end{align*}
are called `leverage scores' in the statistics literature
\cite{ChatterH86,HoagW78,VelleW81}. The largest leverage score
\begin{align*}
\max_{1\leq j\leq p}{\omega_j}=\max_{1\leq j\leq p}
{\sum_{\ell=k+1}^p{v_{j\ell}^2}}
\end{align*}
is called `coherence' in the compressed sensing literature
\cite{DHuo2001}
and reflects the difficulty of sampling rows from 
$\mv_{k+1:p}$.

Method B3 picks a largest leverage score from $\mv_{k+1:p}$,
\begin{align*}
\omega_{m_1}=\max_{1\leq j\leq p}{\omega_j}=\max_{1\leq j\leq p}
{\sum_{\ell=k+1}^p{v_{j\ell}^2}}
\end{align*}
and declares the parameter with index $m_1$  
as unidentifiable.
Method B3 repeats this on the remaining rows of $\mv_{k+1:p}$,
by selecting parameters that have not been selected previously,
\begin{align*}
\omega_{m_{\ell}}=
\max_{\stackrel{k+1\leq j \leq p}{j\neq m_1, \ldots, m_{p-\ell+1}}}{\omega_j}, \qquad \ell=p-1, \ldots, k+1,
\end{align*}
and declares the parameters with index $m_1,\ldots, m_{p-k}$
as unidentifiable.

 \subsubsection*{Expressing PCA method B3 as column subset selection}
PCA method B3 can be interpreted in two ways: Either as
selecting parameters according to the largest leverage scores
of the subdominant eigenvector matrix $\mv_{k+1:p}$
 of the Fischer matrix \cite{ChatterH86,HoagW78,VelleW81};
or else as selecting parameters based on column subset selection with
  \cite[Algorithm GKS-II]{ChI91a}. We choose the latter 
  interpretation.
  
Algorithm~\ref{alg_PCAB3} may look different from PCA method B3 
but accomplishes 
  the same thing in an easier manner (in exact arithmetic). 
  The algorithm in
  \cite[Section 6]{GoKS76}, which is also  
  \cite[Algorithm GKS-I]{ChI91a}, 
  operates instead  the dominant right singular vectors,
  and applies
the  column subset selection method \cite[Section 4]{BuG65}, 
\cite[section 5.4.2]{GovL13},
which is also \cite[Algorithm Golub-I]{ChI91a}.

The idea is the following: partition the SVD
of the triangular matrix
$\mr=\mU_r\msig\mv^T$ in Remark~\ref{r_CH},
\begin{align*}
\msig=\begin{bmatrix}\msig_1& \vzero\\ \vzero & \msig_2\end{bmatrix},
\qquad
\mU_r=\begin{bmatrix} \mU_1 & \mU_2\end{bmatrix}, \qquad
\mv=\begin{bmatrix} \mv_1 & \mv_2\end{bmatrix},
\end{align*}
where $\msig_1=\diag\begin{pmatrix}\sigma_1 & \cdots & \sigma_k\end{pmatrix}\in\real^{k\times k}$ contains the $k$ dominant
singular values  of $\mr$, hence~$\mchi$;
and $\mU_1\in\real^{p\times k}$ and $\mv_1\in\real^{p\times k}$ are
the $k$ associated left and right singular vectors, respectively.
Applying a permutation to $\mv_1^T$ corresponds to applying
a permutation to $\mr$, hence $\mchi$.
In Algorithm~\ref{alg_PCAB3},
column $j$ of $\mw$ is denoted by
$\mw\ve_j$.

 Theorem~\ref{t_PCAB3} quantifies how well the identifiable parameters $\mchi_1$ from Algorithm~\ref{alg_PCAB3} satisfy criterion~(\ref{e_c1}).

\begin{theorem}\label{t_PCAB3}
Let $\mchi\in\real^{n\times p}$ with $n\geq  p$ be the sensitivity matrix, and
$1\leq k<p$. Then Algorithm~\ref{alg_PCAB3} computes
a QR decomposition 
\begin{align*}
\mchi\mP=\begin{bmatrix}\mchi_1 & \mchi_2\end{bmatrix}=
\underbrace{\begin{bmatrix}\mq_1 & \mq_2\end{bmatrix}}_{\mq}
\underbrace{\begin{bmatrix}\mr_{11}&\mr_{12}\\\vzero & \mr_{22}\end{bmatrix}}_{\mr},
\qquad \mr_{11}\in\real^{k\times k},
\end{align*}
where 
\begin{align*}
\sigma_k / \|\mv_{11}^{-1}\|_2\leq   \sigma_k(\mr_{11})  \leq   \sigma_k, \\
\sigma_{k+1} \leq  \sigma_1(\mr_{22}) \leq   \|\mv_{11}^{-1}\|_2 \,\sigma_{k+1}
\end{align*}
and $\mv_{11} \in \real^{k \times k}$ is the leading principal submatrix
of $\mv$ in the SVD $\mr = \mU_r\msig\mv^T$.

If Algorithm~\ref{alg_PCAB3} applies Algorithm~\ref{alg_PCAB4} to $\mv_1^T$, then 
\begin{align*}
\|\mv_{11}^{-1}\|_2\leq 2^{k-1}
\end{align*}
If numerical $\rank(\mchi)= k$, then the columns of $\mchi_1$ represent the $k$ identifiable parameters.
\end{theorem}

\begin{proof}
The upper bound for $\mr_{11}$ and the lower
bound for $\mr_{22}$ follow from interlacing~(\ref{e_inter}).
The remaining two bounds are derived in section~\ref{s_PCAB3p}.

The bound for $\|\mv_{11}^{-1}\|_2=1/\sigma_k(\mv_{11})$ follows by applying Theorem~\ref{t_PCAB4} to  
$\mv_1^T$ and remembering that all singular values 
of $\mv_1$ are equal to 1.
\end{proof}

\begin{algorithm}\caption{Column subset selection version of PCA B3}\label{alg_PCAB3}
\begin{algorithmic}
\REQUIRE $\mchi\in\real^{n\times p}$, $n\geq p$, $1\leq k<p$
\STATE Set $\mP = \mi_p$
\STATE Compute decomposition (\ref{e_pqr}): $\mchi=\mq\mr\qquad$ 
\COMMENT{Unpivoted QR of $\mchi$}
\FOR{$\ell=1:k$} 
\STATE $\qquad$ \COMMENT{If $\ell=1$, then $\mr_{22}=\mr$ }
\STATE Partition 
$\mr=\begin{bmatrix} \mr_{11} & \mr_{12} \\ \vzero& \mr_{22}\end{bmatrix}$
where $\mr_{22}\in\real^{(p-\ell+1)\times (p-\ell+1)}$ \\
$\qquad$ \COMMENT{Focus on trailing $(p-\ell+1)\times (p-\ell+1)$ block}
\STATE Compute $k-\ell+1$ right singular vectors $\mv_1\in\real^{(p-\ell+1)\times (k-\ell+1)}$ of $\mr_{22}$ corresponding to $\sigma_1\geq \cdots\geq \sigma_{k-\ell+1}$
\STATE Set $\mw=\mv_1^T\in\real^{(k-\ell+1)\times (p-\ell+1)}$
\STATE Compute permutation $\widetilde{\mP}\in\real^{(p-\ell+1)\times(p-\ell+1)}$
so that
\STATE $\qquad\qquad\qquad\qquad\|\mw(\widetilde{\mP}\ve_1)\|_2=
\max_{1\leq j\leq p-\ell+1}{\|\mw\ve_j\|_2}$ \\
$\qquad$ \COMMENT{Move column of $\mw$ with largest norm to front}
\STATE Compute QR decomposition (\ref{e_qr}): 
$\mr_{22}\widetilde{\mP}=\widetilde{\mq}\widetilde{\mr}_{22}\qquad$
\COMMENT{Unpivoted QR of $\mr_{22}\widetilde{\mP}$}
\STATE Update 
$\mq:=\mq\begin{bmatrix}\mi_{\ell-1}&\vzero\\ \vzero& \widetilde{\mq} \end{bmatrix}$, 
$\mP:=\mP\begin{bmatrix}\mi_{\ell-1}&\vzero\\ \vzero&\widetilde{\mP}\end{bmatrix}$,
$\mr:=\begin{bmatrix}\mr_{11} & \mr_{12}\\ \vzero& \widetilde{\mr}_{22}
\end{bmatrix}$
\ENDFOR
\RETURN $\mP$, $\mq$, $\mr$
\end{algorithmic}
\end{algorithm}

\subsection{Strong rank-revealing QR decompositions}\label{s_srrqr}
The final method \cite[section 4]{GuE96} selects identifiable parameters by 
trying to maximize the volume of $\ms_1$ via
pairwise column permutations.
%We assume that $k$ is chosen so that $\mr_{11} \in \real^{k \times k}$ in the \textit{unpivoted} QR factorization of $\mchi$ is non-singular.

A  `strong rank-revealing' QR decomposition tries to optimize
both subset selection criteria (\ref{e_c1}) and (\ref{e_c2})
and bounds every element of $|\mr_{11}^{-1} \mr_{12}|$. 
The component-wise boundedness ensures that
the columns of 
\begin{align*}
    \boldsymbol{P} \begin{bmatrix} -\boldsymbol{R}_{11}^{-1} \boldsymbol{R}_{12} \\
    \boldsymbol{I}_{p-k} \end{bmatrix}
\end{align*}
represents an approximate basis for the null space of $\mchi$, provided $\mr_{11}$ is not too ill-conditioned \cite[section 1.2]{GuE96}.
A rigorous definition of the strong rank-revealing  QR decomposition is presented in \cite[Section 1.2]{GuE96} and Theorem~\ref{t_srrqr} below.

Algorithm~\ref{alg_srrqr}, which represents \cite[Algorithm 4]{GuE96},
exchanges a column of $\ms_1$
 with a column of $\ms_2$ until 
 $\det(\mchi_1^T\mchi_1)=\det(\mr_{11})^2$ stops increasing.
More specifically \cite[Lemma 3.1]{GuE96},
after permuting columns $i$ and $k+j$ 
of $\mr$ with a permutation matrix $\mP^{(ij)}$, and
performing an unpivoted QR decomposition $\mchi\mP^{(ij)}=\widetilde{\mq}\widetilde{\mr}$, we compare the determinant of the leading principal submatrix $\widetilde{\mr}_{11}\in\real^{k\times k}$ 
of $\widetilde{\mr}$
with that of the original submatrix $\mr_{11}$,
\begin{align}\label{e_rhoij}
\rho_{ij}\equiv \frac{\det(\widetilde{\mr}_{11})}{\det(\mr_{11})} = 
 \sqrt{(\mr_{11}^{-1} \mr_{12})^2_{ij} + \left(\|\mr_{22}\ve_j\|_2\,\|\ve_i^T\mr_{11}^{-1}\|_2\right)^2}.
\end{align}

Given a user-specified tolerance $f>1$, Algorithm~\ref{alg_srrqr} iterates as
long as it can find columns $i$ and $j+k$ with $\rho_{ij}>f$ and,
by permuting columns $i$ and $j+k$. 
increase the determinant  
to $\det(\widetilde{\mr}_{11})\geq f\det(\mr_{11})$.
The correctness of Algorithm~\ref{alg_srrqr} follows from Lemma~\ref{l_srrqr1}.

Theorem~\ref{t_srrqr} shows that the columns $\mchi_1$ from  Algorithm~\ref{alg_srrqr} can be 
interpreted as identifiable parameters that satisfy even stronger conditions than criteria~(\ref{e_c1}) and (\ref{e_c2}) combined.

\begin{theorem}\label{t_srrqr}
Let $\mchi\in\real^{n\times p}$ with $n\geq  p$ be the sensitivity matrix
and $1\leq k<p$.
%whose leading 
%$k$ columns are linearly
%independent\footnote{This can be accomplished with a 
%preliminary QR decomposition with column %pivoting \cite[Algorithm 5.4.1]{GovL13}.}.
Algorithm~\ref{alg_srrqr} with input
$f\geq 1$ computes a QR decomposition 
\begin{align*}
\mchi\mP=\begin{bmatrix}\mchi_1 & \mchi_2\end{bmatrix}=
\begin{bmatrix}\mq_1 & \mq_2\end{bmatrix}
\begin{bmatrix}\mr_{11}&\mr_{12}\\\vzero & \mr_{22}\end{bmatrix},
\end{align*}
where $\mr_{11}\in\real^{k\times k}$
and $\mr_{22}\in\real^{(p-k)\times(p-k)}$
satisfy
\begin{align*}
\sigma_i(\mr_{11}) & \geq \frac{\sigma_i}{\sqrt{1+f^2k(p-k)}},
\qquad 1 \leq i \leq k\\
\sigma_j(\mr_{22}) & \leq \sigma_{j+k} \sqrt{1+f^2k(p-k)}, \qquad 1 \leq j \leq p-k,
\end{align*}
and
\begin{align*}
|\mr_{11}^{-1}\mr_{12}|_{ij} \leq f,
\qquad 1\leq i\leq k, \ 1\leq j\leq p-k.
\end{align*}
If numerical $\rank(\mchi)= k$, then the columns of $\mchi_1$ represent the $k$ identifiable parameters,
and the columns of $\mchi_2$ the unidentifiable parameters.
\end{theorem}

\begin{proof}
This follows from \cite[Lemma 3.1 and Theorem 3.2]{GuE96}.
See section~\ref{s_srrqrp}, and in particular in Lemma~\ref{l_srrqr}.
\end{proof}

\begin{algorithm}\caption{Column subset selection with strong rank-revealing QR (srrqr)}\label{alg_srrqr}
\begin{algorithmic}
\REQUIRE Sensitivity matrix $\mchi\in\real^{n\times p}$, $n\geq p$,
 $1\leq k<p$, $f \geq 1$
\STATE Compute $\mchi \mP=\mq\mr\qquad$ \COMMENT{Pivoted QR to make $\mr_{11}$  nonsingular}
\STATE Compute $\rho_{ij}$ as defined in (\ref{e_rhoij}), $1\leq i\leq k$, $1\leq j\leq k-p$
\WHILE{$\max_{1\leq i \leq k, \ 1 \leq j \leq p-k}{\{\rho_{ij}\}}> f$}
\STATE Find some $1\leq i\leq k$
and $1\leq j\leq k-p$
with $\rho_{ij}>f$ 
\STATE Compute permutation $\mP^{(ij)}$ to permute columns $i$ and $j+k$ 
\STATE Decomposition (\ref{e_qr}) $\ \mr \mP^{(ij)} = \widetilde{\mq} \widetilde{\mr}\quad$
\COMMENT{Unpivoted QR of $\mr\mP^{(ij)}$}
\STATE Update $\mP:=\mP\mP^{(ij)}$,
$\mq:=\mq\widetilde{\mq}$, 
$\mr:=\widetilde{\mr}$
\STATE Update $\rho_{ij}$ 
\ENDWHILE
\RETURN $\mP$, $\mq$, $\mr$
\end{algorithmic}
\end{algorithm}

 % Applications
\section{Applications}\label{s_appl}
We compare the accuracy of the four
Algorithms 
\ref{alg_PCAB1}--\ref{alg_srrqr}
on the sensitivity matrices from physical applications (section \ref{sec:physmods}) and on the synthetic matrices from classical column pivoting `counterexamples' (section \ref{sec:SynSensDesc}). 

Numerical experiments were performed in MATLAB 2021b on a 16 GB MacBook Pro 
with an M1 chip. We compute relative versions of the 
subset selection criteria (\ref{e_c1}) and (\ref{e_c2}), 
\begin{align}
\label{e_c1rel}
    \gamma_1 \equiv \frac{\sigma_k(\mchi_1)}{\sigma_k(\mchi)},
\end{align}
and
\begin{align}
\label{e_c2rel}
    \gamma_2 \equiv \frac{\|(\mi-\mchi_1 \mchi_1^{\dagger})\mchi_2 \|_2}{\sigma_{k+1}(\mchi)}.
\end{align}
The closer $\gamma_1$ and $\gamma_2$ are to 1, the more accurate
the algorithm. We also compute the improvement in condition number
of the selected columns, 
\begin{align}\label{e_c3c}
\tau\equiv\frac{\cond(\mchi_1)}{\cond(\mchi)}
\end{align}
The lower $\tau_1$, the better the conditioning of the selected columns.

\subsection{Sensitivity matrices from physical models}
\label{sec:physmods}
We apply Algorithms \ref{alg_PCAB1}--\ref{alg_srrqr}
to the sensitivity matrices from 
the mathematical models in sections \ref{sec:RealSensDesc} and~\ref{sec:supp}.

The sensitivity matrices 
$\mchi$ are evaluated at given nominal parameter values. 
For the epidemiological models (SVIR, SEVIR, COVID) in particular,
$\mchi$ is evaluated at the nominal values in Table \ref{tab:epiTab}, and additionally 
at 10,000 points sampled uniformly within 50\% of the nominal value.
  
\subsubsection*{Table \ref{tab:RealTab}}
For each model, Algorithms \ref{alg_PCAB1}--\ref{alg_srrqr} produce the same
identifiable parameters, that is, the
same column subsets and the same identical
values for the subset selection criteria
$\gamma_1$ in (\ref{e_c1rel}) and $\gamma_2$ in (\ref{e_c1rel}).
The consistent accuracy illustrates the robustness of column subset selection for identifiability analysis in  applications, particularly since each sensitivity matrix originates from a different type of mechanistic model.

\begin{table}[htbp]
    \label{tab:RealTab}
    \centering
    \begin{tabular}{c|c|c|c|c|*{2}{r}}
    \hline
      Model  & $n$ & $p$ & $k$ &  $\tau$ & $\gamma_1$  & $\gamma_2$ \\
    \hline
    SVIR & 31 & 4 & 3 & 1.6e-03 & 1.0 & 1.0 \\
    \hline
    
     SEVIR & 31 & 5 & 4 & 1.2e-02 & 1.0 & 1.0 \\
    \hline
     COVID & 31 & 8 & 5 & 1.5e-03 & 0.9 & 1.1 \\
    \hline
     HGO & 14 & 8 & 5 & 4.0e-04 &  1.0 & 1.0 \\
    \hline
     Wound & 46 & 11 & 6 & 2.2e-08 & 0.9 & 1.2 \\
    \hline
     Neuro & 200 & 175 & 14 & 9.8e-23 & 0.6 & 1.7 \\
    \hline
    \end{tabular}
    \caption{Identical accuracy of Algorithms \ref{alg_PCAB1}--\ref{alg_srrqr}
on the models in section \ref{sec:RealSensDesc}. Here
$p=$ number of parameters and number of columns of $\mchi$; $n=$ number of observations and number of  rows of $\mchi$; $k$=  numerical rank of $\mchi$ and number of identifiable  parameters; 
$\tau$= ratio of condition numbers in (\ref{e_c3c}); and $\gamma_1$ and $\gamma_2$ are the
subset selection criteria  in (\ref{e_c1rel}), and in (\ref{e_c2rel}), respectively.}
\end{table}
\smallskip

\begin{center}
\shadowbox{\shortstack{
When applied to the physical models,\\
Algorithms \ref{alg_PCAB1}--\ref{alg_srrqr} 
exhibit similar accuracy and reliability.\\
We recommend Algorithm~\ref{alg_srrqr} because, in theory,\\ 
it has the most stringent accuracy guarantees.}}
\end{center}
\smallskip

\subsection{Synthetic adversarial matrices}
\label{sec:SynSensDesc}
We apply Algorithms \ref{alg_PCAB1}--\ref{alg_srrqr}
to synthetic adversarial matrices designed to thwart the accuracy of subset
selection algorithms. Although synthetic, these matrices
still represent sensitivity matrices for specific
dynamical systems (Appendix~\ref{sec:supp3}).
Each algorithm is
applied to
10,000 realizations of each of the following matrices.

\begin{itemize}
\item \textit{Kahan} \cite{Kah66}: $\ms=\md_n\mk_n\in\real^{n\times n}$, where 
    \begin{align*}
    \md_n \equiv \diag\begin{pmatrix}1 & \zeta & \zeta^2 & \cdots & \zeta^{n-1}
    \end{pmatrix},\qquad
        \mk_n \equiv\begin{pmatrix} 
        1 & -\varphi & -\varphi & \cdots & -\varphi \\
         & 1 &  -\varphi & \cdots & -\varphi  \\
         &  & \ddots & \ddots & \vdots \\
         &  &  & 1 &  -\varphi  \\
         &  &  &  & 1
        \end{pmatrix},
    \end{align*}
with $\zeta^2 + \varphi^2 = 1$ for $\zeta, \varphi>0$, and $k=n-1$.

We choose $n=100$, and sample $\zeta$ uniformly from $[0.9, 0.99999]$.
The average condition number over 10,000 realizations is $\cond(\mchi)\approx 2.4\cdot 10^{19}$.
\smallskip

    \item \textit{Gu-Eisenstat} \cite[Example 2]{GuE96}:
    \begin{align*}
        \ms =\begin{pmatrix} 
        \md_{n-3} \boldsymbol{K}_{n-3} & \boldsymbol{0} & \boldsymbol{0} & -\varphi \md_{n-3} \myone_{n-3} \\
         & \mu & 0 & 0 \\
         & & \mu & 0 \\
         & & & \mu
        \end{pmatrix} \in\rnn,
    \end{align*}
    where $k = n-2$, and 
    \begin{align*}
        \mu \equiv \frac{1}{\sqrt{k}} \min_{1 \leq i \leq n-3} \|\ve_i^T(\md_{n-3} \boldsymbol{K}_{n-3})^{-1}\|_2^{-1}.
    \end{align*}
 We choose $n=100$, and sample $\zeta$ uniformly from $[0.9, 0.99999]$.
The average condition number over 10,000 realizations is $\cond(\mchi)\approx 2.0\cdot 10^{34}$.   
\smallskip
    
        \item \textit{Jolliffe} \cite[Appendix A1]{Joll72}:
    $\ms = \mU \boldsymbol{\Sigma} \mv^T$, where $\mU\in\real^{n\times p}$ has orthonormal columns with Haar measure \cite{Stew80}; $\boldsymbol{\Sigma}\in\real^{p\times p}$ is diagonal; and $\mv$ is the orthonormal factor
    from the QR factorization of 
    \begin{align*}
        \boldsymbol{\Lambda} = \begin{pmatrix} \boldsymbol{\Lambda}_1 &  & &  \\
         & \boldsymbol{\Lambda}_2 &  & \\
          & & \ddots & \\
           &  & & \boldsymbol{\Lambda}_k
          \end{pmatrix}, \qquad  
        \boldsymbol{\Lambda}_i = 
        \begin{pmatrix} 
        1 & \rho_i & \cdots & \rho_i \\
        \rho_i & 1 & \cdots & \rho_i \\
        \vdots & \vdots & \ddots & \vdots \\
        \rho_i & \rho_i & \cdots & 1
        \end{pmatrix}\in\real^{p_i\times p_i},
    \end{align*}
where $\rho_i \approx 1$ and $p = \sum_{i=1}^k{p_i}$.

We choose $n=200$, $p=100$, $p_i=5$, and $k=20$; and  sample the leading $k$ 
diagonal elements of
$\boldsymbol{\Sigma}$ uniformly from $[10^2, 10^3]$, the $p-k$ trailing diagonal
elements of $\boldsymbol{\Sigma}$  uniformly 
from $[10^{-10}, 10^{1.9}]$, and $\rho_i$ uniformly from $[0.9, 0.99999]$. 
The average condition number over 10,000 realizations is 
$\cond(\mchi)\approx 4.8\cdot 10^{14}$.   
\smallskip

 \item \textit{Sorensen-Embree} \cite{SE16}:  $\ms = \mU \boldsymbol{\Sigma} \mv^T$,
 where $\mU\in\real^{n\times p}$ has Haar measure with orthonormal columns; $\boldsymbol{\Sigma}\in\real^{p\times p}$ is diagonal;
 and
  $\mv=\begin{pmatrix} \mv_k&\mv_{p-k}\end{pmatrix}\in\real^{p\times p}$ is an
  orthogonal matrix, and  $\mv_k\in\real^{p\times k}$ is the orthonormal
  factor from the QR factorization of
    \begin{align*}
        \ml = \begin{pmatrix}
        1 &  &  & \\
        -1 & 1 &  & \\
        \vdots & \ddots & \ddots & \\
        -1 & \cdots & -1 & 1 \\
        -1 & \cdots & -1 & -1 \\
        \vdots &  &\vdots & \vdots \\
        -1& \cdots & -1&-1
        \end{pmatrix}\in\real^{p\times k}.
    \end{align*}
We choose $n=200$, $p=100$, and $k=20$; and  sample the leading $k$ 
diagonal elements of
$\boldsymbol{\Sigma}$ uniformly from $[10^2, 10^3]$, the $p-k$ trailing ones
uniformly from $[10^{-10}, 10^{1.9}]$.
The average condition number over 10,000 realizations is 
$\cond(\mchi)\approx 1.4\cdot 10^{14}$.   
\smallskip
    
 \item \textit{SHIPS}:  
We constructed this matrix to amplify differences in 
the accuracy of Algorithms \ref{alg_PCAB1}--\ref{alg_srrqr}. 
Here $\ms = \mU \boldsymbol{\Sigma} \mv^T$,
  where $\mU$ and $\msig$ as for Joliffe, and
  $\mv=\begin{pmatrix} \mv_k&\mv_{p-k}\end{pmatrix}\in\real^{p\times p}$ is an
  orthogonal matrix with 
 \begin{align*}
\mv_k = \begin{pmatrix} \mv_{11}\\ \Tilde{\mU}(\mi-\mv_{11}\mv_{11})^{1/2}\end{pmatrix}\in\real^{p\times k}
    \end{align*}
where $\Tilde{\mU} \in\real^{(p-k)\times k}$ has orthonormal columns with Haar measure \cite{Stew80}, and
    \begin{align*}
    \mv_{11}= \frac{\mt}{2\|\mt\|_2}\in\real^{k\times k}, \qquad
        \boldsymbol{T} = \begin{pmatrix}
        1 & -1  & \cdots & -1 \\
         & 1 & \cdots & -1 \\
         &  & \ddots & \vdots \\
         & & & 1          
        \end{pmatrix}\in\real^{k\times k}.
    \end{align*}
We choose $n=200$, $p=100$, and $k=20$. The leading $k$ 
diagonal elements of $\boldsymbol{\Sigma}$
are logarithmically spaced in $[10^2, 10^3]$, and the $p-k$ trailing ones logarithmically spaced in $[10^{-10}, 10^{1.9}]$. 
The average condition number over 10,000 realizations is 
$\cond(\mchi)\approx 1.0\cdot 10^{13}$.   
\end{itemize} 

In Algorithm~\ref{alg_srrqr}, we set $f=\sqrt{2}$
 for the \textit{Gu-Eisenstat} matrix, and $f=1$ for all other matrices.
 
\subsubsection*{Table~\ref{tab:SynthTab}}
It displays the average of the condition number ratio
(\ref{e_c3c}), and subset selection criteria (\ref{e_c1rel}) and (\ref{e_c2rel}) for 10,000 realizations of each synthetic matrix. 

Algorithm~\ref{alg_PCAB4} produces the smallest values
of $\tau$ and $\gamma_1$, that is, the worst 
conditioned columns $\mchi_1$, for the
\textit{Kahan} and \textit{Gu-Eisenstat} matrices.

Algorithm~\ref{alg_PCAB3} produces the smallest values
of $\gamma_2$, that is, the best low-rank approximation $\mchi_1$.
Algorithms \ref{alg_PCAB1} and \ref{alg_srrqr} are close 
with only slightly larger $\gamma_2$ on all matrices
except for the \textit{Sorensen-Embree} matrix,
where their $\gamma_2$ is more than 5 times larger than
that of Algorithm~\ref{alg_PCAB3}. 

Algorithms \ref{alg_PCAB1} and \ref{alg_srrqr} 
produce better conditioned $\mchi_1$
than Algorithm~\ref{alg_PCAB3}, most notably for
the \textit{Sorensen-Embree} and \textit{SHIPS} matrices. 

The \textit{Jolliffe} matrix was constructed to thwart 
Algorithm~\ref{alg_PCAB3} \cite[Appendix A1]{Joll72}, and there is slight evidence of its loss of accuracy with these matrices.
While all of the algorithms performed nearly identically, the absolute version of criterion~(\ref{e_c1}) for Algorithm~\ref{alg_PCAB3} (to more digits than could be represented in Table~\ref{tab:SynthTab}) is 1.8e-14, compared to 1.9e-14 for Algorithms~\ref{alg_PCAB1}, \ref{alg_PCAB4}, and \ref{alg_srrqr}.

\subsubsection*{Figure~\ref{fig:SHIPShist}} 
The box plots illustrate the accuracy
of Algorithms \ref{alg_PCAB1}--\ref{alg_srrqr} on 10,000 realizations of our
\textit{SHIPS} matrix. The top and bottom of each box represent the first and third quartiles, respectively, while the red line through the box itself is the average. 
Values below and above the short black horizontal lines are outliers,
and the horizontal lines themselves show the minimum and maximum excluding the outliers.

We constructed the SHIPS matrix to force differences in the accuracy
of Algorithms \ref{alg_PCAB1}--\ref{alg_srrqr}.
It illustrates the superior accuracy 
of Algorithm~\ref{alg_srrqr} in the conditioning 
(\ref{e_c3c}) of the selected columns $\mchi_1$,
as well as subset selection criteria (\ref{e_c1rel}) and (\ref{e_c2rel}). 

\subsubsection*{Figure~\ref{fig:SHIPShist}(a)} 
Algorithm~\ref{alg_srrqr} gives the best,
that is smallest, ratio of condition numbers. In contrast,
Algorithms \ref{alg_PCAB3} and \ref{alg_PCAB4} have
a larger number of outliers above the maximum,  illustrating more 
less reliable accuracy.

\subsubsection*{Figure~\ref{fig:SHIPShist}(b)} 
Algorithm~\ref{alg_srrqr}
gives the best, that is closest to 1,
values of $\gamma_1$.
In contrast, Algorithm \ref{alg_PCAB3}
has more outliers below its minimum,
indicating less reliable accuracy.

\subsubsection*{Figure~\ref{fig:SHIPShist}(c)}
Algorithm~\ref{alg_srrqr} has the most consistent values of $\gamma_2$, but 
they are slightly larger than
those for Algorithms \ref{alg_PCAB3} and \ref{alg_PCAB4}.
Its maximum and the outliers above are comparable to those of \ref{alg_PCAB3}.
In contrast, Algorithm ~\ref{alg_PCAB1} is much less accurate. 

While there are differences among Algorithms \ref{alg_PCAB1}--\ref{alg_srrqr} they are relatively small, 
suggesting that all are effective in practice. However, we still recommend Algorithm~\ref{alg_srrqr} since it is numerically stable, computationally efficient, and is the only one 
whose bounds do not depend exponentially on $p$ or $k$.

\begin{table}[htbp]
\label{tab:SynthTab}
\centering
\begin{tabular}{c|lr*{3}{r}}
$\mchi$  & Algorithms & $\tau$ & $\gamma_1$  & $\gamma_2$ \\
\hline
 & & & & & \\ 
 %%%% check/fix everything except for GuEis
 \texttt{Kahan} & \begin{tabular}{@{}l@{}} \ref{alg_PCAB1}, \ref{alg_srrqr} \\ \ref{alg_PCAB4} \\ \ref{alg_PCAB3} \end{tabular} & \begin{tabular}{@{}r@{}} \textbullet \ 3.7\texttt{e}-03 \\ 6.4\texttt{e}-01 \\ \textbullet \ 3.7\texttt{e}-03 \end{tabular} &
\begin{tabular}{@{}r@{}} \textbullet \ 1.0 \\ 1.6\texttt{e}-03 \\ \textbullet \ 1.0 \end{tabular}  & \begin{tabular}{@{}r@{}} 1.8\texttt{e}03 \\ 1.9\texttt{e}15 \\ \textbullet \ 1.7\texttt{e}03 \end{tabular}  \\
 & & & & & \\
 
  \texttt{GuEis} & \begin{tabular}{@{}l@{}} \ref{alg_PCAB1}, \ref{alg_srrqr} \\ \ref{alg_PCAB4} \\ \ref{alg_PCAB3} \end{tabular} & \begin{tabular}{@{}r@{}} \ \ 4.1\texttt{e}-03 \\ \ \ 4.1\texttt{e}-03 \\ \ \ 4.1\texttt{e}-03  \end{tabular} &
\begin{tabular}{@{}r@{}} 0.6 \\ 0.6 \\ 0.6 \end{tabular}  & \begin{tabular}{@{}r@{}} 0.9 \\ 5.2\texttt{e}11 \\ \textbullet \ 1.0 \end{tabular}  \\
 & & & & & \\

\texttt{Joll} & \ref{alg_PCAB1},\ref{alg_PCAB4}, \ref{alg_PCAB3}, \ref{alg_srrqr} & \ \ 1.6\texttt{e}-12 & 1.0 & 1.0  \\
 & & & & & \\
 
\texttt{SorEm} & \begin{tabular}{@{}l@{}}\ref{alg_PCAB1}, \ref{alg_srrqr} \\ \ref{alg_PCAB4} \\ \ref{alg_PCAB3} \end{tabular}  & \begin{tabular}{@{}r@{}} \textbullet \ 1.4\texttt{e}-12 \\ 2.2\texttt{e}-12 \\ 2.3\texttt{e}-12 \end{tabular}  & \begin{tabular}{@{}r@{}}\textbullet \ 0.9 \\ 0.5 \\ 0.5   \end{tabular}  & \begin{tabular}{@{}r@{}}  5.4 \\ 1.1 \\ \textbullet \ 1.0     \end{tabular} \\
 & & & & & \\
 
  \texttt{SHIPS} & \begin{tabular}{@{}l@{}} \ref{alg_PCAB1} \\ \ref{alg_PCAB4} \\ \ref{alg_PCAB3} \\ \ref{alg_srrqr} \end{tabular} & \begin{tabular}{@{}r@{}} 1.9\texttt{e}-12 \\
2.9\texttt{e}-12 \\
2.0\texttt{e}-12 \\
\textbullet \ 1.6\texttt{e}-12  \end{tabular}  & \begin{tabular}{@{}r@{}} 0.3 \\ 0.2 \\ 0.3 \\
 \textbullet \ 0.4 \end{tabular}  & \begin{tabular}{@{}r@{}} 2.4 \\ 1.4 \\ \textbullet 1.4 \\ 1.9 \end{tabular}  \\
 & & & & & \\
 \hline
\end{tabular}
\caption{Accuracy of Algorithms~\ref{alg_PCAB1}--\ref{alg_srrqr} on the synthetic
matrices. For each matrix $\mchi$, the average condition number
ratio $\tau$ in (\ref{e_c3c}), and the average subset selection criteria  $\gamma_1$ in (\ref{e_c1rel}) and $\gamma_2$
in (\ref{e_c2rel}) over 10,000 realizations are displayed.
A \textbullet \  denotes an optimal value for the corresponding criterion.}
\end{table}

%\II{In the top line, please reply the condition number ratio by $\tau$. 
%On the horizontal axes, please replace algorithm names by numbers, see Table 5.1 }

\begin{figure}[htbp]
    \centering
    \includegraphics[width=\textwidth]{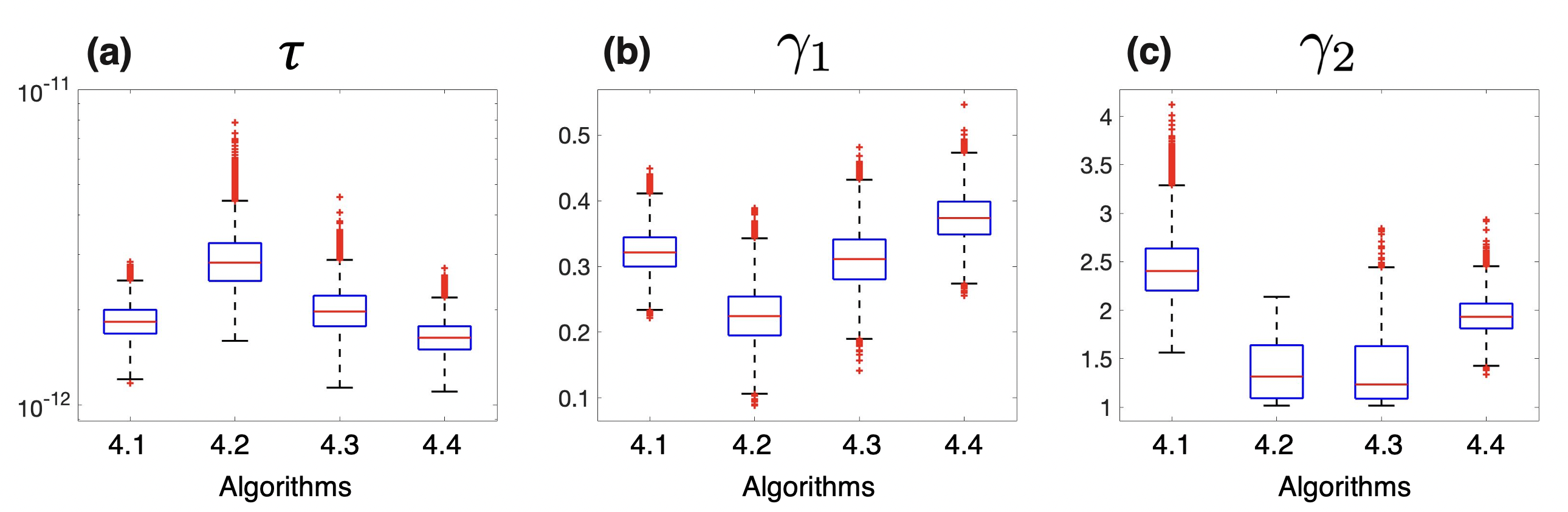}
    \caption{Application of Algorithms \ref{alg_PCAB1}--\ref{alg_srrqr} to
  10,000 realizations of the \textup{\texttt{SHIPS}} matrix.   
    Box plots show (a) the ratio of condition numbers $\tau$
    in (\ref{e_c3c}), and the subset selection criteria (b) $\gamma_1$ in (\ref{e_c1rel}), and (c) $\gamma_2$ in (\ref{e_c2rel}).}
    \label{fig:SHIPShist}
\end{figure}

\section{Conclusion}
We have presented a numerically accurate and reliable approach for practical parameter identifiability analysis in the context of physical models.

Our recommendation is to perform column subset selection (CSS) 
directly on the sensitivity matrix~$\mchi$, rather than detouring through the error-prone formation of the Fischer matrix $\mf=\mchi^T\mchi$ followed by an eigenvalue decomposition. 

We applied the four CSS Algorithms \ref{alg_PCAB1}--\ref{alg_srrqr},
to a large variety of
practical and adversarial sensitivity matrices, and 
they produced almost identical sets of identifiable
parameters $\mchi_1$ with vastly improved condition numbers compared to 
the condition number of the original matrix $\mchi$.

The superior accuracy of CSS is important 
when identifiability analysis is part of a larger  application.
In the context of inverse problems, 
for instance, parameters designated as unidentifiable
may be fixed at a nominal value, for the purpose
of dimension reduction.
If this is an iterative process, reliable  designation
of unidentifiable parameters is important.

\subsubsection*{Future research}
We discuss several avenues for future research, many of which will necessitate 
challenging modifications to Algorithms \ref{alg_PCAB1}--\ref{alg_srrqr}.

\begin{enumerate}
\item Efficient implementation of Algorithms \ref{alg_PCAB1}--\ref{alg_srrqr}.\\
This includes the  
choice of QR decompositions and data structure; as well as 
fast updates, searches for magnitude-largest elements, and computation of $k$.

\item Application of CSS methods to pharmacology.\\
Physiologically-based pharmacokinetic (PBPK) and quantitative systems pharmacology (QSP) models
exhibit moderate- to high-dimensional parameter spaces with highly nonlinear dependencies in their ODEs.  For example, the minimal brain PBPK model in \cite{Bloomingdale2021} has as many as 37 parameters in 16 coupled ODEs.  This requires that unidentifiable parameters be determined and fixed
at nominal values at the very start --prior to optimization, sensitivity analysis, Bayesian inference for computing parameter distributions, and uncertainty propagation for constructing prediction intervals for QoIs.

Another difficulty is the optimization of criteria (\ref{e_c1rel})--(\ref{e_c3c})
for larger QSP and PBPK models, as they  
may depend strongly on the number $n$ of observations, 
the number $p$ of parameters, and the number $k$ of identifiable parameters.

\item Global CSS algorithms.\\
Algorithms \ref{alg_PCAB1}--\ref{alg_srrqr} are local in the sense that they operate on a single set of nominal parameter values.  However, there is significant motivation in the PBPK and QSP communities to identify parameter dependencies for a \textit{range} of admissible parameter values.  
Although it
might be tempting to simply average the sensitivity values, in the manner of active subspace analysis \cite{Constantine2015}, 
the highly nonlinear nature of parameter dependencies tends to rule out this approach.

\item Mixed effects.\\
Another challenge in PBPK and QSP models are the regimes that combine
both, population and individual attributes. This necessitates mixed-effects models, which try to quantify the fixed-effects due to population parameters on the one hand; and the distributions for random effects associated with individuals on the other.  A first step would be to incorporate
CSS methods into the initial parameter subset selection algorithm for mixed-effects models in \cite{Schmidt_Smith2016}.  

\item Virtual populations.\\
A broad area of research in QSP models concerns the
generation of virtual populations for the purpose of safe and efficient drug development  \cite{ARM2016}.  This requires the perturbation of QSP models about nominal values and characterization of sensitivities and uncertainties associated with model parameters.  We anticipate that the CSS algorithms will play an increasing role in this growing field of virtual population generation and selection.
\end{enumerate}

\appendix
\section{Proofs}\label{s_addr}
We present the proofs of
Theorem~\ref{t_PCAB1} (section~\ref{s_PCAB1p}),
Theorem~\ref{t_PCAB4} (section~\ref{s_PCAB4p}),
Theorem~\ref{t_PCAB3} (section~\ref{s_PCAB3p}),
and Theorem~\ref{t_srrqr} (section~\ref{s_srrqrp}).

Let $\mchi\in\real^{n\times p}$ be the
sensitivity matrix with $n\geq p$,
singular values $\sigma_1\geq \cdots\geq \sigma_p\geq 0$,
and a pivoted QR decomposition,
partitioned for some $1\leq k<p$ so that 
\begin{align*}
\mchi\mP=\mq\begin{bmatrix}\mr_{11} & \mr_{12}\\ \vzero & \mr_{22}\end{bmatrix},
\qquad \mr_{11}\in\real^{k\times k}, \quad \mr_{22}\in\real^{(p-k)\times (p-k)}.
\end{align*}
Singular value interlacing \cite[Corollary 8.6.3]{GovL13}
implies that the singular values of $\mr_{11}$ cannot
exceed the corresponding dominant singular values of $\mchi$, while 
the singular values of $\mr_{22}$ cannot be smaller than the corresponding
subdominant singular values of $\mchi$, that is,
\begin{align}\label{e_inter}
\begin{split}
\sigma_j(\mr_{11}&)\leq \sigma_j, \qquad 1\leq j\leq k \\
\sigma_j(\mr_{22}) &\geq \sigma_{k+j}, \qquad 1\leq j\leq p-k.
\end{split}
\end{align}

\subsection{Proof of Theorem~\ref{t_PCAB1}}\label{s_PCAB1p}
We present an approximation for the smallest singular
value (Lemma~\ref{l_PCAB1}), a correctness proof Algorithm~\ref{alg_PCAB1} (Lemma~\ref{l_PCAB1b}),
and a proof of Theorem~\ref{t_PCAB1}
(Lemma~\ref{l_PCAB1tp}).

In the subsequent proofs we combine different bits and pieces from
\cite[sections 7 and 8]{ChI91a} and
\cite[section 3]{Chan1987}, and add more details for comprehension.

The key observation is that a judiciously chosen permutation can reveal a 
smallest singular value in a diagonal element of the triangular matrix 
in a QR decomposition.
Below is a consequence of a more general statement in
\cite[Theorem 2.1]{Chan1987}.

\begin{lemma}[Revealing a smallest singular value]\label{l_PCAB1}
Let $\vv$ with $\|\vv\|_2=1$ be a right singular vector of
$\mb\in\real^{m\times m}$ associated with a smallest 
singular value $\sigma_m(\mb)$, so that $\|\mb\vv\|_2=\sigma_m(\mb)$.
 Let $\mP\in\real^{m\times m}$ be a permutation that moves a magnitude-largest 
element of $\vv$ to the bottom, $|(\mP^T\vv)_m|=\|\vv\|_{\infty}$.
If $\mb\mP=\mq\mr$ is an unpivoted QR decomposition (\ref{e_qr}) of $\mb\mP$, 
then 
the trailing diagonal element of the upper triangular matrix~$\mr$ satisfies
\begin{align*}
\sigma_m(\mb)\leq |r_{mm}|\leq \sqrt{m}\,\sigma_m(\mb).
\end{align*}
\end{lemma}

\begin{proof}
The lower bound follows from singular value interlacing (\ref{e_inter}).
As for the upper  bound, the relation between the right singular 
vector $\vv$ and a corresponding left singular vector $\vu$ 
with $\mb\vv=\sigma_m(\mb)\vu$ and $\|\vu\|_2=1$ implies
\begin{align*}
\sigma_m(\mb)\vu=\mb\vv=(\mb\mP)\,(\mP^T\vv)=\mq\mr\,(\mP^T\vv)=
\mq \mr \begin{bmatrix}*\\ (\mP^T\vv)_m\end{bmatrix}
\end{align*}
From this, $\|\vu\|_2=1$, the unitary invariance 
of the two-norm, and the upper triangular nature
of $\mr$ follows \begin{align*}
\sigma_m(\mb)=\|\sigma_m(\mb)\vu\|_2=\|\mr(\mP^T\vv)\|_2\geq |r_{mm}(\mP^T\vv)_m|
= |r_{mm}|\,\|\vv\|_{\infty}\geq |r_{mm}|/\sqrt{m}.
\end{align*}
The last inequality follows from the fact that $\vv\in\real^m$ has
unit two-norm $\|\vv\|_2=1$, so at least one of its $m$ elements must be sufficiently
large with $\|\vv\|_{\infty}\geq 1/\sqrt{m}$.
\end{proof}

\begin{lemma}[Correctness of Algorithm~\ref{alg_PCAB1}]\label{l_PCAB1b}
Let $\mchi\in\real^{n\times p}$ with $n\geq p$ have singular values
$\sigma_1\geq \cdots \geq \sigma_p\geq 0$, and pick some $1\leq k<p$. Then
Algorithm~\ref{alg_PCAB1}  computes a QR decomposition 
$\mchi\mP=\mq\mr$
where the $p-k$ trailing diagonal elements of $\mr$ satisfy
\begin{align*}
|\mr_{\ell\ell}|\leq \sqrt{\ell}\,\sigma_{\ell}, \qquad k+1\leq \ell\leq p.
\end{align*}
\end{lemma}

\begin{proof}
This is an induction proof on the iterations~$i$ of Algorithm~\ref{alg_PCAB1}
with more discerning notation.
The initial pivoted decomposition reduces the problem size
\begin{align}\label{e_PCAa}
\mchi\mP^{(0)}=\mq^{(0)}\mr^{(0)},
\end{align}
where $\mP^{(0)}\in\real^{p\times p}$ is a permutation, $\mq^{(0)}\in\real^{n \times p}$
has orthonormal columns, and $\mr^{(0)}\in\real^{p\times p}$ is upper triangular.

\paragraph{Induction basis}
Set $\mr_{11}^{(1)}=\mr^{(0)}\in\real^{p\times p}$, and let $\vv^{(1)}, \vu^{(1)}\in\real^p$
be right and left singular vectors associated with a smallest singular value, 
\begin{align*}
\mr_{11}^{(1)}\vv^{(1)}=\sigma_p\vu^{(1)}, \qquad \|\vv^{(1)}\|_2=\|\vu^{(1)}\|_2=1.
\end{align*}
Determine a permutation $\widetilde{\mP}^{(1)}$ that moves a magnitude-largest element 
of $\vv^{(1)}$ to the bottom, 
\begin{align*}
|((\widetilde{\mP}^{(1)})^T\vv^{(1)})_p|=\|\vv^{(1)}\|_{\infty}\geq 1/\sqrt{p}.
\end{align*}
Compute an unpivoted QR decomposition 
$\mr_{11}^{(1)}\widetilde{\mP}^{(1)}=\widetilde{\mq}^{(1)}\widetilde{\mr}_{11}^{(1)}$,
where $\widetilde{\mq}^{(1)}\in\real^{p\times p}$ is an orthogonal matrix.
Lemma~\ref{l_PCAB1} implies that  the trailing diagonal element of the
triangular matrix reveals a smallest singular value,
$|(\widetilde{\mr}_{11}^{(1)})_{pp}|\leq \sqrt{p}\,\sigma_p$.
Insert this into the initial decomposition~(\ref{e_PCAa})
\begin{align*}
\mchi\mP^{(0)}=\mq^{(0)}\mr^{(0)}=\mq^{(0)}
\widetilde{\mq}^{(1)}\widetilde{\mr}_{11}^{(1)}(\widetilde{\mP}^{(1)})^T.
\end{align*}
Multiply by $\widetilde{\mP}^{(1)}$ on the right,
\begin{align*}
\mchi\,\underbrace{\mP^{(0)}\widetilde{\mP}^{(1)}}_{\mP^{(1)}}=
\underbrace{\mq^{(0)}\widetilde{\mq}^{(1)}}_{\mq^{(1)}}\,
\underbrace{\widetilde{\mr}_{11}^{(1)}}_{\mr^{(1)}} \qquad
\text{where} \qquad |\mr^{(1)}_{pp}|\leq \sqrt{p}\sigma_p.
\end{align*}

\paragraph{Induction hypothesis}
Assume that $\mchi\mP^{(i)}=\mq^{(i)}\mr^{(i)}$ for $i=p-\ell$ and $\ell>k+1$ with
\begin{align*} 
|\mr^{(i)}_{jj}|\leq \sqrt{j}\,\sigma_{j}, \qquad \ell\leq j \leq p.
\end{align*}

\paragraph{Induction step}
Here $\ell=k+2$ is the dimension of the leading block, while $i\equiv p-\ell$
is the dimension of the trailing block. Partition
\begin{align}\label{e_PCAd}
\mr^{(i)} =\begin{bmatrix}\mr_{11}^{(i)} &\mr^{(i)}_{12}\\ \vzero & \mr^{(i)}_{22}
\end{bmatrix}\qquad \mr^{(i)}_{11}\in\real^{\ell\times \ell}, \quad
\mr^{(i)}_{22}\in\real^{i\times i}.
\end{align}
Let $\vv^{(i+1)}, \vu^{(i+1)}\in\real^{\ell}$
be  right and left singular vectors associated with a smallest singular value
of $\mr_{11}^{(i)}$,
\begin{align}\label{e_PCAc}
\mr_{11}^{(i)}\vv^{(i+1)}=\sigma_{\ell}(\mr^{(i)}_{11})\,\vu^{(i+1)}, \qquad 
\|\vv^{(i+1)}\|_2=\|\vu^{(i+1)}\|_2=1.
\end{align}
Determine a permutation $\widetilde{\mP}^{(i+1)}$ that moves a magnitude-largest element 
of $\vv^{(i+1)}$ to the bottom, 
\begin{align*}
|((\widetilde{\mP}^{(i+1)})^T\vv^{(i+1)})_{\ell}|=\|\vv^{(i+1)}\|_{\infty}\geq 1/\sqrt{\ell}.
\end{align*}
Compute an unpivoted QR decomposition 
$\mr_{11}^{(i)}\widetilde{\mP}^{(i+1)}=\widetilde{\mq}^{(i+1)}\widetilde{\mr}_{11}^{(i+1)}$,
where $\widetilde{\mq}^{(i+1)}\in\real^{\ell\times \ell}$ is an orthogonal matrix.
Lemma~\ref{l_PCAB1} implies that  the trailing diagonal element of the
triangular matrix reveals a smallest singular value,
\begin{align}\label{e_PCAb}
|(\widetilde{\mr}_{11}^{(i+1)})_{\ell \ell}|\leq 
\sqrt{\ell}\,\sigma_{\ell}(\mr^{(i)}_{11}).
\end{align}
Insert this into the decomposition $\mchi\mP^{(i)}=\mq^{(i)}\mr^{(i)}$
with partitioning (\ref{e_PCAd}), and exploit the fact that the 
inverse of the orthogonal matrix 
$\widetilde{\mq}^{(i+1)}$ is $(\widetilde{\mq}^{(i+1)})^T$, 
\begin{align*}
\mchi\mP^{(i)}=\mq^{(i)}\mr^{(i)}
&=\mq^{(i)}\begin{bmatrix}
\widetilde{\mq}^{(i+1)}\widetilde{\mr}_{11}^{(i+1)}(\widetilde{\mP}^{(i+1)})^T &
\mr_{12}^{(i)}\\ \vzero & \mr_{22}^{(i)}\end{bmatrix}\\
&=\mq^{(i)}
\begin{bmatrix}\widetilde{\mq}^{(i+1)} & \vzero\\ \vzero& \mi_i\end{bmatrix}
\begin{bmatrix}\widetilde{\mr}_{11}^{(i+1)} &(\widetilde{\mq}^{(i+1)})^T \mr_{12}^{(i)}\\ 
\vzero & \mr_{22}^{(i)}\end{bmatrix}
\begin{bmatrix} (\widetilde{\mP}^{(i+1)})^T&\vzero\\ \vzero& \mi_i\end{bmatrix}
\end{align*}
Multiply by the permutation on the right,
\begin{align*}
\mchi\underbrace{\mP^{(i)}\begin{bmatrix} \widetilde{\mP}^{(i+1)}&\vzero\\ 
\vzero& \mi_i\end{bmatrix}}_{\mP^{(i+1)}}
&=\underbrace{\mq^{(i)}\begin{bmatrix}\widetilde{\mq}^{(i+1)} & \vzero\\ \vzero& 
\mi_i\end{bmatrix}}_{\mq^{(i+1)}}
\underbrace{\begin{bmatrix}\widetilde{\mr}_{11}^{(i+1)} &
(\widetilde{\mq}^{(i+1)})^T \mr_{12}^{(i)}\\ 
\vzero & \mr_{22}^{(i)}\end{bmatrix}}_{\mr^{(i+1)}}.
\end{align*}
From (\ref{e_PCAb}), interlacing (\ref{e_inter}), and the fact that $\mr^{(i)}$ has the 
same singular values as~$\mchi$ follows
\begin{align*}
 |(\mr^{(i+1)})_{\ell \ell}|=|(\widetilde{\mr}_{11}^{(i+1)})_{\ell \ell}| \leq 
\sqrt{\ell}\,\sigma_{\ell}(\mr^{(i)}_{11})\leq \sqrt{\ell}\,\sigma_{\ell}(\mr^{(i)})
= \sqrt{\ell}\,\sigma_{\ell}.
\end{align*}
Together with the induction hypothesis, and $i=p-\ell=p-(k+2)$ this implies
\begin{align*}
 |\mr^{(p-k+1)}_{jj}| \leq \sqrt{j}\,\sigma_{j}, \qquad k+1\leq j\leq p.
\end{align*}
\end{proof}

\begin{lemma}[Proof of Theorem~\ref{t_PCAB1}]\label{l_PCAB1tp}
Let $\mchi\in\real^{n\times p}$ with $n\geq p$ have singular
values $\sigma_1\geq \cdots \geq \sigma_p\geq 0$, and pick some $1\leq k<p$.
Then Algorithm~\ref{alg_PCAB1} computes a QR decomposition 
\begin{align*}
\mchi\mP=\begin{bmatrix}\mq_1 & \mq_2\end{bmatrix}
\begin{bmatrix}\mr_{11}&\mr_{12}\\\vzero & \mr_{22}\end{bmatrix},
\end{align*}
where the largest singular value of $\mr_{22}\in\real^{(p-k)\times (p-k)}$ is bounded by
\begin{align*}
\|\mr_{22}\|_2&\leq p\,\|\mw^{-1}\|_2\, \sigma_{k+1}.
\end{align*}
Here $\mw\in\real^{(p-k)\times (p-k)}$ is a triangular matrix with diagonal
elements $|w_{jj}|=1$, $1\leq j\leq p-k$; offdiagonal elements $|w_{ij}|\leq 1$
for $i\neq j$;  and
\begin{align*}
\|\mw^{-1}\|_2\leq 2^{p-k-1}.
\end{align*}
\end{lemma}

\begin{proof}
Let $\mchi\mP=\mq\mr$ be computed by Algorithm~\ref{alg_PCAB1} with 
input $k$. The proof is an extension of Lemma~\ref{l_PCAB1}.
From the right singular vectors in Algorithm~\ref{alg_PCAB1}
we construct a matrix $\mz$, and then bound $\|\mr\mz\|_2$ 
to derive an upper bound for $\|\mr_{22}\|_2$.

\paragraph{Construction of $\mz$}
The indexing of the partition is different than the one in~(\ref{e_PCAd}), 
\begin{align*}
\mr^{(\ell)} =\begin{bmatrix}\mr_{11}^{(\ell)} &\mr^{(\ell)}_{12}\\ \vzero & \mr^{(\ell)}_{22}
\end{bmatrix}\qquad \mr^{(\ell)}_{11}\in\real^{\ell\times \ell}, \quad
\mr^{(\ell)}_{22}\in\real^{(p-\ell)\times (p-\ell)}, \qquad k+1\leq \ell\leq p.
\end{align*}
In the statement of this lemma, the partitioning is $\ell=k$.

Let $\vv^{(\ell)}, \vu^{(\ell)}\in\real^{\ell}$ be right and left singular vectors associated
with a smallest singular value of $\mr_{11}^{(\ell)}$,
\begin{align*}
\mr_{11}^{(\ell)}\vv^{(\ell)}=\sigma_{\ell}(\mr_{11}^{(\ell)})\,\vu^{(\ell)},\qquad 
\|\vv^{(\ell)}\|_2=\|\vu^{(\ell)}\|_2=1, \qquad k+1\leq \ell\leq p.
\end{align*}
Algorithm~\ref{alg_PCAB1} has permuted the right singular vectors so that
a magnitude-largest element is at  the bottom,
\begin{align}\label{e_PCAf}
|\vv^{(\ell)}_{\ell}|\geq 1/\sqrt{\ell}
\quad \text{and}\quad  |\vv^{(\ell)}_j|\leq |\vv^{(\ell)}_{\ell}|, \qquad
1\leq j<\ell, \quad k+1\leq \ell\leq p.
\end{align}
The trailing elements in singular vectors associated with larger-dimensional
blocks are not affected by subsequent permutations,
see  (\ref{e_PCAd}), where permutations in the $(1,1)$ block do not
affect the $(2,2)$ block and its placement of diagonal elements.

Construct an upper trapezoidal  matrix 
$\mz=\begin{bmatrix}\vz_1 & \cdots &\vz_{p-k}\end{bmatrix}\in\real^{p\times (p-k)}$,
whose columns are the right singular vectors 
\begin{align*}
\vz_{\ell-k}=\begin{bmatrix} \vv^{(\ell)} \\ \vzero_{p-\ell}\end{bmatrix}, \qquad
k+1\leq \ell\leq p.
\end{align*}
Factor out the diagonal elements and focus on the trailing $(p-k)\times (p-k)$ submatrix
\begin{align}\label{e_PCAh}
\mz=\begin{bmatrix}\mz_1 \\ \mw\end{bmatrix}\md, \qquad \text{where}\qquad
\md=\begin{bmatrix}v_{k+1}^{(k+1)} && \\ &\ddots & \\ && v_p^{(p)}\end{bmatrix}
\in\real^{(p-k)\times (p-k)}
\end{align}
has diagonal elements 
$|d_{\ell\ell}|=|v_{\ell}^{(\ell)}|\geq 1/\sqrt{\ell}$, $k+1\leq \ell\leq p$.
From (\ref{e_PCAf}) follows that  $\mw\in\real^{(p-k)\times (p-k)}$
is a nonsingular upper triangular matrix with elements
\begin{align*}
|w_{\ell\ell}|=1, \qquad |w_{\ell j}|\leq 1, \qquad 1\leq \ell\leq p-k, \quad j>\ell.
\end{align*}

\paragraph{Bounds for $\|\mr\mz\|_2$}
We derive an upper and a lower bound.
Multiplying the QR decomposition $\mchi\mP=\mq\mr$ by $\mq^T$ on the left and 
by $\mz$ on the right gives 
\begin{align*}
\mq^T\mchi\mP\mz=\mr\mz\in\real^{p-k}.
\end{align*}
The columns of $\mr\mz$ are
\begin{align*}
\mr\vz_{\ell-k}=\begin{bmatrix} \mr_{11}^{(\ell)}\vv^{(\ell)}\\ \vzero_{p-\ell}\end{bmatrix}=
\sigma_{\ell}(\mr_{11}^{(\ell)})\begin{bmatrix}\vu^{(\ell)}\\\vzero_{p-\ell}\end{bmatrix},
\qquad k+1\leq \ell\leq p.
\end{align*}
From $\|\vu^{(\ell)}\|_2=1$ and interlacing (\ref{e_inter}) follows
\begin{align*}
\|\mr\vz_{\ell-k}\|_2=\sigma_{\ell}(\mr_{11}^{(\ell)})\leq\sigma_{\ell},\qquad 
k+1\leq \ell\leq p.
\end{align*}
Bound the norm of $\mr\mz\in\real^{p\times (p-k)}$ in terms of its largest 
column norm \cite[section 2.3.2]{GovL13} to obtain the upper bound
\begin{align}\label{e_PCAe}
\|\mr\mz\|_2\leq \sqrt{p-k}\,\max_{k+1\leq \ell\leq p}{\|\mr\vz_{\ell-k}\|_2} \leq
\sqrt{p-k}\,\max_{k+1\leq \ell\leq p}{\sigma_{\ell}}\leq \sqrt{p}\,\sigma_{k+1}.
\end{align}
As for the lower bound, use the partitioning in the statement of this lemma,
\begin{align*}
\mr\mz=\begin{bmatrix}\mr_{11} &\mr_{12}\\ \vzero & \mr_{22}
\end{bmatrix}\begin{bmatrix}\mz_1\md\\ \mw\md\end{bmatrix}=
\begin{bmatrix}\mr_{11}\mz_1\md+\mr_{12}\mw\md\\
\mr_{22}\mw\md\end{bmatrix},
\end{align*}
and bound $\|\mr\mz\|_2$ in terms of the trailing component 
\begin{align*}
\|\mr\mz\|_2\geq \|\mr_{22}\mw\md\|_2\geq 
\frac{\|\mr_{22}\|_2}{\|\mw^{-1}\|_2\|\md^{-1}\|_2}
\geq\frac{\|\mr_{22}\|_2}{\sqrt{p}\,\|\mw^{-1}\|_2}.
\end{align*}
At last combine the above upper bound with the lower bound (\ref{e_PCAe}), 
\begin{align*}
\|\mr_{22}\|\leq p\,\|\mw^{-1}\|_2\,\sigma_{k+1}.
\end{align*}
The bound for $\|\mw^{-1}\|_2$ is derived in \cite[Theorem 8.14]{Higham2002}; and there 
are classes of matrices for which it can essentially be tight \cite[section 8.3]{Higham2002}.
\end{proof}

\subsection{Proof of Theorem~\ref{t_PCAB4}}\label{s_PCAB4p}
We present an approximation for the largest singular
value (Lemma~\ref{l_PCAB4}), a correctness proof Algorithm~\ref{alg_PCAB4} (Lemma~\ref{l_PCAB4b}),
and a proof of Theorem~\ref{t_PCAB4}
(Lemma~\ref{l_PCAB4tp}).

In the subsequent proofs, we present more general and simpler derivations 
than the 
ones in \cite[section 7]{ChI91a} and \cite[sections 2 and 3]{ChanHansen}, 
and add more details for comprehension.

The key observation is that a judiciously chosen permutation can reveal a largest
singular value in a diagonal element of the triangular matrix in a QR decomposition.
The next statement represents part of 
\cite[Theorem 2.1]{ChanHansen}, however with a 
simpler proof that does not require 
a pseudo inverse as in \cite[Theorems 6.1 and 6.2]{ChanHansen}.

\begin{lemma}[Revealing a largest singular value]\label{l_PCAB4}
Let $\vv$ with $\|\vv\|_2=1$ be a right singular vector of
$\mb\in\real^{m\times m}$ associated with a largest
singular value $\sigma_1(\mb)$, so that $\|\mb\vv\|_2=\sigma_1(\mb)$.
 Let $\mP\in\real^{m\times m}$ be a permutation that moves a magnitude-largest 
element of $\vv$ to the top, $|(\mP^T\vv)_1|=\|\vv\|_{\infty}$.
If $\mb\mP=\mq\mr$ is an unpivoted QR decomposition (\ref{e_qr}) of~$\mb\mP$, 
then 
the leading diagonal element of the upper triangular matrix~$\mr$ satisfies
\begin{align*}
\sigma_1(\mb)/\sqrt{m}\leq |r_{11}|\leq \sigma_1 (\mb).
\end{align*}
\end{lemma}

\begin{proof}
The upper bound follows from singular value interlacing (\ref{e_inter}).
As for the lower bound, the relation between the right singular 
vector $\vv$ and a corresponding left singular vector $\vu$ with 
$\mb^T\vu=\sigma_1(\mb)\vv$ and $\|\vu\|_2=1$  implies
\begin{align*}
\sigma_1(\mb)\,\mP^T\vv=\mP^T\mb\,\vu=\mr^T\mq^T\vu.
\end{align*}
From this,  the lower triangular nature of $\mr^T$, the Cauchy Schwartz inequality, 
and $\|\vu\|_2=1$ follows for the leading element 
\begin{align*}
\sigma_1(\mb)\|\vv\|_{\infty}=|\sigma_1(\mb)(\mP^T\vv)_1|=
|\ve_1^T\mr^T(\mq^T\vu)|\leq \|\mr\ve_1\|_2\|\mq^T\vu\|_2=|r_{11}|.
\end{align*}
Then $\|\vv\|_{\infty}\geq 1/\sqrt{m}$
follows from the fact that $\vv\in\real^m$ has unit two-norm $\|\vv\|_2=1$, so at least 
one of its $m$ elements must be sufficiently large.
\end{proof}

\begin{lemma}[Correctness of Algorithm~\ref{alg_PCAB4}]\label{l_PCAB4b}
Let $\mchi\in\real^{n\times p}$ with $n\geq p$ have singular values
$\sigma_1\geq \cdots\geq \sigma_p\geq 0$, and pick some $1\leq k<p$. Then
Algorithm~\ref{alg_PCAB4}  computes a QR decomposition 
$\mchi\mP=\mq\mr$
where the $k$ leading diagonal elements of $\mr$ satisfy
\begin{align*}
\sigma_{\ell}/\sqrt{p-\ell+1}\leq |\mr_{\ell\ell}|, \qquad 1\leq \ell\leq k.
\end{align*}
\end{lemma}

\begin{proof}
This is an induction proof on the iterations~$\ell$ of Algorithm~\ref{alg_PCAB4}
with more discerning notation.
The initial pivoted decomposition reduces the problem size
\begin{align}\label{e_PCA4a}
\mchi\mP^{(0)}=\mq^{(0)}\mr^{(0)},
\end{align}
where $\mP^{(0)}\in\real^{p\times p}$ is a permutation, $\mq^{(0)}\in\real^{n \times p}$
has orthonormal columns, and $\mr^{(0)}\in\real^{p\times p}$ is upper triangular.

\paragraph{Induction basis}
Set $\mr_{22}^{(1)}=\mr^{(0)}\in\real^{p\times p}$, and let $\vv^{(1)}, \vu^{(1)}\in\real^p$
be right and left singular vectors associated with a largest singular value, 
\begin{align*}
\mr_{22}^{(1)}\vv^{(1)}=\sigma_1\vu^{(1)}, \qquad \|\vv^{(1)}\|_2=\|\vu^{(1)}\|_2=1.
\end{align*}
Determine a permutation $\widetilde{\mP}^{(1)}$ that moves a magnitude-largest element 
of $\vv^{(1)}$ to the top, 
\begin{align*}
|((\widetilde{\mP}^{(1)})^T\vv^{(1)})_1|=\|\vv^{(1)}\|_{\infty}\geq 1/\sqrt{p}.
\end{align*}
Compute an unpivoted QR decomposition 
$\mr_{22}^{(1)}\widetilde{\mP}^{(1)}=\widetilde{\mq}^{(1)}\widetilde{\mr}_{22}^{(1)}$,
where $\widetilde{\mq}^{(1)}\in\real^{p\times p}$ is an orthogonal matrix.
Lemma~\ref{l_PCAB4} implies that  the leading diagonal element of the
triangular matrix reveals a largest singular value,
$|(\widetilde{\mr}_{22}^{(1)})_{11}|\geq \sigma_1/\sqrt{p}$.
Insert this into the initial decomposition~(\ref{e_PCA4a})
\begin{align*}
\mchi\mP^{(0)}=\mq^{(0)}\mr^{(0)}=\mq^{(0)}
\widetilde{\mq}^{(1)}\widetilde{\mr}_{22}^{(1)}(\widetilde{\mP}^{(1)})^T.
\end{align*}
Multiply by $\widetilde{\mP}^{(1)}$ on the right,
\begin{align*}
\mchi\,\underbrace{\mP^{(0)}\widetilde{\mP}^{(1)}}_{\mP^{(1)}}=
\underbrace{\mq^{(0)}\widetilde{\mq}^{(1)}}_{\mq^{(1)}}\,
\underbrace{\widetilde{\mr}_{22}^{(1)}}_{\mr^{(1)}} \qquad
\text{where} \qquad |\mr^{(1)}_{22}|\geq \sigma_1/\sqrt{p}.
\end{align*}

\paragraph{Induction hypothesis}
Assume that $\mchi\mP^{(\ell)}=\mq^{(\ell)}\mr^{(\ell)}$ for $\ell<k$ with
\begin{align*} 
|\mr^{(\ell)}_{jj}|\geq \sigma_{j}/\sqrt{p-j+1}, \qquad 1\leq j\leq \ell.
\end{align*}

\paragraph{Induction step}
Here $\ell= k-1$. The dimension 
of the leading block is $\ell-1$, while the dimension of the trailing block
is $i\equiv p-(\ell-1)$. Partition
\begin{align}\label{e_PCA4d}
\mr^{(\ell)} =\begin{bmatrix}\mr_{11}^{(\ell)} &\mr^{(\ell)}_{12}\\ \vzero & \mr^{(\ell)}_{22}
\end{bmatrix}\qquad \mr^{(\ell)}_{11}\in\real^{(\ell-1)\times (\ell-1)}, \quad
\mr^{(\ell)}_{22}\in\real^{i\times i}.
\end{align}
Let $\vv^{(\ell+1)}, \vu^{(\ell+1)}\in\real^{i}$
be  right and left singular vectors associated with a largest singular value
of $\mr_{22}^{(\ell)}$,
\begin{align}\label{e_PCA4c}
\mr_{22}^{(\ell)}\vv^{(\ell+1)}=\sigma_{1}(\mr^{(\ell)}_{22})\,\vu^{(\ell+1)}, \qquad 
\|\vv^{(\ell+1)}\|_2=\|\vu^{(\ell+1)}\|_2=1.
\end{align}
Determine a permutation $\widetilde{\mP}^{(\ell+1)}\in\real^{i\times i}$ 
that moves a magnitude-largest element of $\vv^{(\ell+1)}$ to the top, 
\begin{align*}
|((\widetilde{\mP}^{(\ell+1)})^T\vv^{(\ell+1)})_{1}|=\|\vv^{(\ell+1)}\|_{\infty}\geq 1/\sqrt{i}.
\end{align*}
Compute an unpivoted QR decomposition 
$\mr_{22}^{(\ell)}\widetilde{\mP}^{(\ell+1)}=
\widetilde{\mq}^{(\ell+1)}\widetilde{\mr}_{22}^{(\ell+1)}$,
where $\widetilde{\mq}^{(\ell+1)}\in\real^{i\times i}$ is an orthogonal matrix.
Lemma~\ref{l_PCAB4} implies that  the leading diagonal element of the
triangular matrix reveals a largest singular value,
\begin{align}\label{e_PCA4b}
|(\widetilde{\mr}_{22}^{(\ell+1)})_{11}|\geq \sigma_1(\mr^{(\ell)}_{22})/\sqrt{i}.
\end{align}
Insert this into the decomposition $\mchi\mP^{(\ell)}=\mq^{(\ell)}\mr^{(\ell)}$
with partitioning (\ref{e_PCA4d}), and exploit the fact that the inverse of the
orthogonal matrix $\widetilde{\mq}^{(\ell+1)}$ equals
$(\widetilde{\mq}^{(\ell+1)})^T$,
\begin{align*}
\mchi\mP^{(\ell)}=\mq^{(\ell)}\mr^{(\ell)}
&=\mq^{(\ell)}\begin{bmatrix}\mr_{11}^{(\ell)} & \mr_{12}^{(\ell)} \\ \vzero& 
\widetilde{\mq}^{(\ell+1)}\widetilde{\mr}_{22}^{(\ell+1)}(\widetilde{\mP}^{(\ell+1)})^T 
\end{bmatrix}\\
&=\mq^{(\ell)}
\begin{bmatrix}\mi_{\ell-1} & \vzero \\ \vzero & \widetilde{\mq}^{(\ell+1)} \end{bmatrix}
\begin{bmatrix}\mr_{11}^{(\ell)} & \mr_{12} ^{(\ell)}\\ \vzero&
\widetilde{\mr}_{22}^{(\ell+1)} \end{bmatrix}
\begin{bmatrix} \mi_{\ell-1} & \vzero\\ \vzero &(\widetilde{\mP}^{(\ell+1)})^T\end{bmatrix}
\end{align*}
Multiply by the permutation on the right,
\begin{align*}
\mchi\underbrace{\mP^{(\ell)}\begin{bmatrix}\mi_{\ell-1}& \vzero\\ 
\vzero & \widetilde{\mP}^{(\ell+1)}\end{bmatrix}}_{\mP^{(\ell+1)}}
&=\underbrace{\mq^{(\ell)}\begin{bmatrix}\mi_{\ell-1} & \vzero\\
\vzero& \widetilde{\mq}^{(\ell+1)}\end{bmatrix}}_{\mq^{(\ell+1)}}
\underbrace{\begin{bmatrix}\mr_{11}^{(\ell)} & \mr_{12}^{(\ell)}\\ 
\vzero & \widetilde{\mr}_{22}^{(\ell+1)}\end{bmatrix}}_{\mr^{(\ell+1)}}.
\end{align*}
From (\ref{e_PCA4b}), interlacing (\ref{e_inter}), and the fact that $\mr^{(\ell)}$ has the 
same singular values as~$\mchi$ follows
\begin{align*}
 |\mr^{(\ell+1)}_{\ell\ell}|=|(\widetilde{\mr}_{22}^{(\ell+1)})_{11}| \geq 
\sigma_1(\mr^{(\ell)}_{22})/\sqrt{i}\geq \sigma_{\ell}(\mr^{(\ell)})/\sqrt{i}
= \sigma_{\ell}/\sqrt{i}.
\end{align*}
Together with the induction hypothesis, and $\ell=k-1$ this implies
\begin{align*}
 |\mr^{(k)}_{jj}| \geq \sigma_{j}/\sqrt{p-j+1}, \qquad 1\leq j\leq k.
\end{align*}
\end{proof}

\begin{lemma}[Proof of Theorem~\ref{t_PCAB4}]\label{l_PCAB4tp}
Let $\mchi\in\real^{n\times p}$ with $n\geq p$ have singular 
values $\sigma_1\geq \cdots \geq \sigma_p\geq 0$, and pick some $1\leq k<p$.
Then Algorithm~\ref{alg_PCAB4} computes a QR decomposition 
\begin{align*}
\mchi\mP=\begin{bmatrix}\mq_1 & \mq_2\end{bmatrix}
\begin{bmatrix}\mr_{11}&\mr_{12}\\\vzero & \mr_{22}\end{bmatrix},
\end{align*}
where the smallest singular value of $\mr_{11}\in\real^{k\times k}$ is bounded by
\begin{align*}
\sigma_k(\mr_{11})\geq \frac{\sigma_k}{p\,\|\mw^{-1}\|_2}.
\end{align*}
Here $\mw\in\real^{k\times k}$ is a triangular matrix with diagonal
elements $|w_{jj}|=1$, $1\leq j\leq k$; offdiagonal elements $|w_{ij}|\leq 1$ 
for $i\neq j$;  and
\begin{align*}
\|\mw^{-1}\|_2\leq 2^{k-1}.
\end{align*}
\end{lemma}

\begin{proof}
Let $\mchi\mP=\mq\mr$ be computed by Algorithm~\ref{alg_PCAB4} with 
input $k$. The proof is an extension of Lemma~\ref{l_PCAB4}, and
is more general than the one in \cite[section 7]{ChI91a} due to the absence
of inverses and no need for the requirement $\sigma_k>0$.

From the right singular vectors in Algorithm~\ref{alg_PCAB4}
we construct a matrix $\mz$, and also a matrix $\my$ of left singular vectors.
Then we bound the $k$th singular value of a top submatrix of 
$\mr^T\my$, to derive a lower  bound for $\sigma_k(\mr_{11})$.

\paragraph{Construction of $\mz$ and $\my$}
Consider the partitionings as in~(\ref{e_PCA4d}) with $i\equiv p-(\ell-1)$
\begin{align*}
\mr^{(\ell)} =\begin{bmatrix}\mr_{11}^{(\ell)} &\mr^{(\ell)}_{12}\\ \vzero & \mr^{(\ell)}_{22}
\end{bmatrix}\qquad \mr^{(\ell)}_{11}\in\real^{(\ell-1)\times (\ell-1)}, \quad
\mr^{(\ell)}_{22}\in\real^{i\times i}, \quad 1\leq \ell\leq k.
\end{align*}
In the statement of this lemma, the partitioning is $\ell=k+1$.

Let $\vv^{(\ell)}, \vu^{(\ell)}\in\real^i$ be right and left singular vectors associated
with a largest singular value of $\mr_{22}^{(\ell)}$,
\begin{align*}
\mr_{22}^{(\ell)}\vv^{(\ell)}=\sigma_1(\mr_{22}^{(\ell)})\,\vu^{(\ell)},\qquad 
\|\vv^{(\ell)}\|_2=\|\vu^{(\ell)}\|_2=1, \qquad 1\leq \ell\leq k.
\end{align*}
Algorithm~\ref{alg_PCAB4} has permuted the right singular vectors so that
a magnitude-largest element is at  the top, for $1\leq \ell\leq k$
\begin{align}\label{e_PCA4f}
|\vv^{(\ell)}_1|\geq 1/\sqrt{i}
\quad \text{and}\quad  |\vv^{(\ell)}_j|\leq |\vv^{(\ell)}_1|, \qquad 1< j\leq i.
\end{align}
The leading elements in singular vectors associated with larger-dimensional
blocks are not affected by subsequent permutations,
see  (\ref{e_PCA4d}), where permutations in the $(2,2)$ block do not
affect the $(1,1)$ block and its placement of diagonal elements.

Construct a lower trapezoidal  matrix 
$\mz=\begin{bmatrix}\vz_1 & \cdots &\vz_k\end{bmatrix}\in\real^{p\times k}$,
whose columns are the right singular vectors 
\begin{align*}
\vz_{\ell}=\begin{bmatrix} \vzero_{\ell-1}\\ \vv^{(\ell)}\end{bmatrix}, \qquad
1\leq \ell\leq k.
\end{align*}
Factor out the diagonal elements and distinguish the leading $k\times k$ submatrix
\begin{align}\label{e_PCA4h}
\mz=\begin{bmatrix}\mw \\ \mz_2\end{bmatrix}\md, \qquad \text{where}\qquad
\md=\begin{bmatrix}v_1^{(1)} && \\ &\ddots & \\ && v_1^{(k)}\end{bmatrix}
\in\real^{k\times k}
\end{align}
has diagonal elements 
$|d_{\ell\ell}|=|v_1^{(\ell)}|\geq 1/\sqrt{p-\ell+1}$, $1\leq \ell\leq k$.
From (\ref{e_PCA4f}) follows that  $\mw\in\real^{k\times k}$
is a nonsingular lower triangular matrix with elements
\begin{align*}
|w_{\ell\ell}|=1, \qquad |w_{j\ell}|\leq 1, \qquad 1\leq\ell\leq k, \quad j>\ell.
\end{align*}
Analogously, construct a second lower trapezoidal  matrix 
$\my=\begin{bmatrix}\vy_1 & \cdots &\vy_k\end{bmatrix}\in\real^{p\times k}$,
whose columns are the right left vectors 
\begin{align*}
\vy_{\ell}=\begin{bmatrix} \vzero_{\ell-1}\\ \vu^{(\ell)}\end{bmatrix}, \qquad
\|\vy_{\ell}\|_2=1,\qquad 1\leq \ell\leq k,
\end{align*}
and distinguish the leading $k\times k$ submatrix
\begin{align}\label{e_PCA4j}
\my=\begin{bmatrix}\my_1 \\ \my_2\end{bmatrix}, \qquad \text{where}\qquad
\my_1\in\real^{k\times k},\qquad \|\my_1\|_2\leq \sqrt{k}.
\end{align}

\paragraph{Bounds for $\sigma_k(\mr_{11}^T\my_1)$}
We derive an upper and a lower bound.

The columns of $\mr^T\my$ are for $1\leq \ell\leq k$,
\begin{align*}
\mr^T\vy_{\ell}=\begin{bmatrix}(\mr_{11}^{(\ell)})^T & \vzero\\
(\mr_{12}^{(\ell)})^T & (\mr_{22}^{(\ell)})^T\end{bmatrix}
\begin{bmatrix}\vzero_{\ell-1}\\ \vu_{\ell}\end{bmatrix}=
\begin{bmatrix}\vzero_{\ell-1}\\  (\mr_{22}^{(\ell)})^T\vu_{\ell}\end{bmatrix}=
\begin{bmatrix} \vzero_{\ell_1} \\ \sigma_1(\mr_{22}^{\ell})\vv_{\ell}\end{bmatrix}
=\sigma_1(\mr_{22}^{\ell})\vz_{\ell}.
\end{align*}
Collecting all the columns gives
\begin{align*}
\mr^T\my=\mz\mdel\qquad \text{where}\quad
\mdel=\begin{bmatrix} \sigma_1(\mr_{22}^{(1)}) && \\ &\ddots & \\ &&
\sigma_1(\mr_{22}^{(k)})\end{bmatrix}\in\real^{k\times k}.
\end{align*}
With the partitioning of $\mr$ as in the statement of this lemma, the top 
$k\times k$ submatrix of $\mr^T\my=\mz\mdel$ equals
\begin{align*}
\mr_{11}^T\my_1=\mw\md\mdel.
\end{align*}
First derive the lower bound from the right side. The Weyl product inequalities 
\cite[7.3.P16]{HoJ12} imply
\begin{align}\label{e_PCA4e}
\sigma_k(\mr_{11}^T\my_1)=\sigma_k(\mw\md\mdel)\geq \sigma_k(\mw)\,
\sigma_k(\md)\,\sigma_k(\mdel)\geq
\frac{\sigma_k}{\sqrt{p-k+1}\,\|\mw^{-1}\|_2}
\end{align}
where the last inequality follows from applying interlacing (\ref{e_inter}) to
\begin{align*}
\sigma_k(\mdel)=\min_{1\leq\ell\leq k}{\sigma_1(\mr_{22}^{(\ell)})}\geq \sigma_k,
\end{align*}
and bounding the diagonal elements of $\md$ in (\ref{e_PCA4h}) by 
\begin{align*}
\sigma_k(\md)=\min_{1\leq \ell\leq k}{|v_1^{(\ell)}|}\geq 1/\sqrt{p-k+1}.
\end{align*}
Now derive the lower bound from the left side. The Weyl product inequalities 
\cite[7.3.P16]{HoJ12} and (\ref{e_PCA4j}) imply
\begin{align*}
\sigma_k(\mr_{11}^T\my_1)\leq \sigma_k(\mr_{11})\|\my_1\|_2
\leq \sqrt{k}\sigma_k(\mr_{11}).
\end{align*}
At last, combine this with (\ref{e_PCA4e}) to obtain
\begin{align*}
\sigma_k(\mr_{11})\geq \frac{\sigma_k}{\sqrt{k(p-k+1)}\,\|\mw^{-1}\|_2}
\geq \frac{\sigma_k}{p\,\|\mw^{-1}\|_2}.
\end{align*}
The bound for $\|\mw^{-1}\|_2$ follows as in the proof of Lemma~\ref{l_PCAB1tp}.
\end{proof}

\subsection{Proof of Theorem~\ref{t_PCAB3}}
\label{s_PCAB3p}
The following is an extension of \cite[Theorem 5.5.2]{GovL13}.

\begin{lemma}[Proof of Theorem~\ref{t_PCAB3}]
Let $\mchi\in\real^{n\times p}$ with $n\geq p$
have singular values $\sigma_1\geq \cdots \geq\sigma_p\geq 0$, and pick some $1\leq k<p$. 
If Algorithm~\ref{alg_PCAB3} computes a QR decomposition 
\begin{align*}
\mchi\mP=\begin{bmatrix}\mq_1 & \mq_2\end{bmatrix}
\begin{bmatrix}\mr_{11}&\mr_{12}\\\vzero & \mr_{22}\end{bmatrix},
\end{align*}
and chooses the permutation $\mP$ so that  $\mv_{11}\in\real^{k\times k}$ is nonsingular, 
then $\mr_{11}\in\real^{k\times k}$ and
$\mr_{22}\in\real^{(p-k)\times (p-k)}$ satisfy
\begin{align*}
\sigma_k/\|\mv_{11}^{-1}\|_2\leq \sigma_k(\mr_{11})\leq \sigma_k\\
\sigma_{k+1}\leq \sigma_1(\mr_{22})\leq 
\|\mv_{11}^{-1}\|_2\,\sigma_{k+1}.
\end{align*}
 \end{lemma}
 
\begin{proof}
Let $\mchi=\mq\mr$
be a preliminary  unpivoted QR decomposition, where $\mq\in\real^{n\times p}$ has orthonormal columns, and $\mr\in\real^{p\times p}$
is upper triangular. Then let $\mr=\mU_r\msig\mv^T$ be an SVD of the triangular matrix as in Remark~\ref{r_CH}.
Distinguish
the matrix of $k$ largest singular values
$\msig_1\in\real^{k\times k}$ of $\mchi$, and the corresponding right singular 
vectors $\mv_1\in\real^{p\times k}$,
\begin{align*}
\msig=\begin{bmatrix}\msig_1 & \vzero\\ \vzero &\msig_2\end{bmatrix}\in\real^{p\times p},
\qquad 
\mv=\begin{bmatrix}\mv_1 &\mv_2\end{bmatrix}\in\real^{p\times p}.
\end{align*}

\paragraph{Main idea}
Perform a QR decomposition with column pivoting on $\mv_1^T$,
\begin{align*}
\mv_1^T\mP=\mq_1\begin{bmatrix}\mv_{11} & \mv_{12}\end{bmatrix},
\end{align*}
where $\mP\in\real^{p\times p}$ is a permutation matrix, 
$\mq_1\in\real^{k\times k}$ is an orthogonal matrix, 
and $\mv_{11}\in\real^{k\times k}$ is nonsingular upper triangular.
Partition commensurately,
\begin{align*}
\mv_2^T\mP=\begin{bmatrix}\mv_{21}& \mv_{22}\end{bmatrix},
\end{align*}
where $\mv_{22}\in\real^{(p-k)\times (p-k)}$.
Express the permuted upper triangular matrix $\mr\mP$ in terms of these partitions,
\begin{align}\label{e_rp}
\begin{split}
\mr\mP&=\mU_r\msig\mv^T\mP=
\mU_r\begin{bmatrix}\msig_1 & \vzero\\ \vzero & \msig_2\end{bmatrix}
\begin{bmatrix}\mq_1 & \vzero\\ \vzero&\mi\end{bmatrix}\begin{bmatrix}\mv_{11} & \mv_{12}\\ \mv_{21} & \mv_{22}\end{bmatrix}\\
&=\mU_r\begin{bmatrix}\widehat{\msig}_1 & \vzero\\ \vzero & \msig_2\end{bmatrix}
\begin{bmatrix}\mv_{11} & \mv_{12}\\ \mv_{21} & \mv_{22}\end{bmatrix}\qquad
\text{where}\quad
\widehat{\msig}_1\equiv\msig_1\mq_1.
\end{split}
\end{align}
Because $\mq_1$ is an orthogonal matrix, 
$\widehat{\msig}_1$ has the same singular values as $\msig_1$, that is,
\begin{align}\label{e_sigmaj}
\sigma_j(\widehat{\msig}_1)=\sigma_j(\msig_1)=\sigma_j,\qquad 1\leq j\leq k.
\end{align}
Re-triangularize by computing an unpivoted QR decomposition of $\mr\mP$,
\begin{align}\label{e_rpqr}
\mr\mP=\mq_r\begin{bmatrix} \mr_{11} & \mr_{12} \\ \vzero & \mr_{22}\end{bmatrix}
\end{align}
where $\mr_{11}\in\real^{k\times k}$ is upper triangular.

\paragraph{Inequality for $\mr_{11}$}
Equate (\ref{e_rpqr}) with 
(\ref{e_rp}) and move $\mU_r$ to the left
\begin{align*}
\mU_r^T\mq_r\begin{bmatrix} \mr_{11} & \mr_{12} \\ \vzero & \mr_{22}\end{bmatrix}
=\mU_r^T\mr\mP=
\begin{bmatrix}\widehat{\msig}_1 & \vzero\\ \vzero & \msig_2\end{bmatrix}
\begin{bmatrix}\mv_{11} & \mv_{12}\\ \mv_{21} & \mv_{22}\end{bmatrix}.
\end{align*}
The goal is to extract $\mr_{11}$. To this end partition 
\begin{align*}
\mU_r^T\mq_r=\begin{bmatrix}\mU_{11} & \mU_{12}\\ \mU_{21} & \mU_{22}\end{bmatrix}
\end{align*}
and substitute this into the above expression for $\mU_r^T\mr\mP$,
\begin{align*}
\begin{bmatrix}\mU_{11} & \mU_{12}\\ \mU_{21} & \mU_{22}\end{bmatrix}
\begin{bmatrix} \mr_{11} & \mr_{12} \\ \vzero & \mr_{22}\end{bmatrix}
=
\begin{bmatrix}\widehat{\msig}_1 & \vzero\\ \vzero & \msig_2\end{bmatrix}
\begin{bmatrix}\mv_{11} & \mv_{12}\\ \mv_{21} & \mv_{22}\end{bmatrix}.
\end{align*}
Due to the triangular and diagonal matrices, the (1,1) block of this equation is 
\begin{align*}
\mU_{11}\mr_{11}=\widehat{\msig}_1\mv_{11}.
\end{align*}
Apply the Weyl product inequalities for singular values \cite[(7.3.14)]{HoJ12}
to the smallest singular value of the matrices on both sides and remember
(\ref{e_sigmaj}),
\begin{align*}
\frac{\sigma_k}{\|\mv_{11}^{-1}\|_2}&=
\sigma_k(\widehat{\msig}_1)\sigma_k(\mv_{11})\leq \sigma_k(\widehat{\msig}_1\mv_{11})=
\sigma_k(\mU_{11}\mr_{11}).
\end{align*}
Because the orthogonal matrix $\mU$ has all singular values equal to one,
\begin{align*}
\sigma_k(\mU_{11}\mr_{11})\leq
\sigma_1(\mU_{11})\sigma_k(\mr_{11})\leq
\sigma_1(\mU)\sigma_k(\mr_{11})=\sigma_k(\mr_{11}).
\end{align*}
Combining the extreme ends of the sequence of inequalities gives
$\sigma_k/\|\mv_{11}^{-1}\|_2\leq \sigma_k(\mr_{11})$.

\paragraph{Inequality for $\mr_{22}$}
Again,
equate (\ref{e_rpqr}) with (\ref{e_rp}) but now move the $\mv$ matrix to the left, 
\begin{align*}
\begin{bmatrix} \mr_{11} & \mr_{12} \\ \vzero & \mr_{22}\end{bmatrix}
\begin{bmatrix}\mv_{11}^T & \mv_{21}^T\\ 
\mv_{12}^T & \mv_{22}^T\end{bmatrix}
=\begin{bmatrix}\mU_{11}^T & \mU_{21}^T\\ \mU_{12}^T & \mU_{22}^T\end{bmatrix}
\begin{bmatrix}\widehat{\msig}_1 & \vzero\\ \vzero & \msig_2\end{bmatrix}.
\end{align*}
As before, the triangular and diagonal matrices imply that the (2,2) block of this equation is 
\begin{align*}
\mr_{22}\mv_{22}^T=\mU_{22}^T\msig_2.
\end{align*}
Apply the Weyl product inequalities for singular values \cite[(7.3.14)]{HoJ12}
to the largest singular value of the matrices on both sides,
\begin{align*}
\frac{\sigma_1(\mr_{22})}{\|\mv_{22}^{-1}\|_2}
= \sigma_1(\mr_{22})\sigma_{p-k}(\mv_{22})
\leq 
\sigma_1(\mr_{22}\mv_{22}^T)=
\sigma_1(\mU_{22}^T\msig_2).
\end{align*}
Because the orthogonal matrix $\mU$ has all singular values equal to one,
\begin{align*}
\sigma_1(\mU_{22}^T\msig_2)\leq 
\sigma_1(\mU_{22})\sigma_1(\msig_2)\leq
\sigma_1(\mU)\sigma_{k+1}=\sigma_{k+1}.
\end{align*}

Since $\mv_{11}$ is nonsingular,  the CS decomposition \cite[Theorem 2.5.3]{GovL13}
 implies that  $\|\mv_{11}^{-1}\|_2=\|\mv_{22}^{-1}\|_2$.
 Combining the extreme ends of the sequence of the above inequalities gives
 $\sigma_1(\mr_{22})/\|\mv_{11}^{-1}\|_2\leq \sigma_{k+1}$.
\end{proof}

\subsection{Proof of Theorem~\ref{t_srrqr}}\label{s_srrqrp}

We prove the correctness of Algorithm~\ref{alg_srrqr}
(Lem\--ma~\ref{l_srrqr1}), and present a proof of Theorem~\ref{t_srrqr} (Lemma~\ref{l_srrqr}).

Our proofs follow those in \cite{GuE96} but without the full rank assumption on the sensitivity 
matrix and with more details.
To keep the proofs simple, we assume that the QR decompositions are implemented 
so that the upper triangular matrices have non-negative diagonal elements
\cite[Theorem~5.2.3]{GovL13}.

We prove the correctness of stopping criterion
of Algorithm~\ref{alg_srrqr}, which depends on
the row norms of $\mr_{11}^{-1}$ and the column norms of~$\mr_{22}$,
\begin{align*}
\omega_i(\mr_{11})&\equiv 1/\|\ve_i^T\mr_{11}^{-1}\|_2, \qquad 1\leq i\leq k\\
\gamma_j(\mr_{22})&\equiv \|\mr_{22}\ve_j\|_2,\qquad 1\leq j\leq p-k.
\end{align*}

\begin{lemma}[Correctness of Algorithm~\ref{alg_srrqr}]
\label{l_srrqr1}
Let 
\begin{align*}
    \mr = \begin{bmatrix} \mr_{11} & \mr_{12} \\ \vzero & \mr_{22} \end{bmatrix}\in\real^{p\times p}
\end{align*}
be upper triangular with non-negative
diagonal elements,  nonsingular
$\mr_{11} \in \real^{k \times k}$, and $\mr_{22} \in \real^{(p-k) \times (p-k)}$.
Let $\mP$ be a permutation that permutes columns $i$ and $k+j$ of $\mr$ for 
some $1 \leq i \leq k$ and some $1 \leq j \leq p-k$, and let
$\mr \mP=\widetilde{\mq} \widetilde{\mr}$ be an 
unpivoted QR decomposition with
\begin{align*}
\widetilde{\mr} =\begin{bmatrix} \widetilde{\mr}_{11} & \widetilde{\mr}_{12} \\ \vzero & \widetilde{\mr}_{22} \end{bmatrix}.
\end{align*}
Then 
\begin{align*}
 \rho_{ij}\equiv   \frac{\det(\widetilde{\mr}_{11})}{\det(\mr_{11})} = \sqrt{(\mr_{11}^{-1} \mr_{12})^2_{i,j} + (\gamma_j(\mr_{22})/\omega_i(\mr_{11}))^2}.
\end{align*}
\end{lemma}
 
 \begin{proof} 
We give the proof for the special case $i = k$ and $j=1$, and first
argue that this represents no loss of generality.
Note that column~$j$ of $\mr_{22}$ corresponds to column~$k+j$ of $\mr$.

\paragraph{Reduction to the case $i=k$ and $j=1$}
Suppose that  $i < k$ and $j > 1$. Let $\mP_{i,k}$ be the permutation that permutes columns $i$ and $k$ of $\mr$, and let $\mr_{11} \mP_{i,k} = \bar{\mq}_{11} \bar{\mr}_{11}$ be the unpivoted QR decomposition.
Similarly, let $\mP_{1,j}$ be the permutation that permutes columns $k+j$ and $k+1$ of $\mr$, and
let $\mr_{22} \mP_{1,j}=\bar{\mq}_{22}\bar{\mr}_{22}$ be the unpivoted QR decomposition. 
With
\begin{align*}
\bar{\mr}_{12} \equiv  \bar{\mq}_{11}^T \mr_{12} \mP_{1,j}, \qquad
\bar{\mP} \equiv \begin{bmatrix} \mP_{i,k} & \vzero \\  \vzero & \mP_{1,j} \end{bmatrix},
\end{align*}
the matrix
\begin{align*}
    \mr \bar{\mP} = \begin{bmatrix} \mr_{11} & \mr_{12} \\ \vzero & \mr_{22} \end{bmatrix} \begin{bmatrix} \mP_{i,k} & \vzero \\ \vzero  & \mP_{1,j} \end{bmatrix} = \begin{bmatrix} \mr_{11}\mP_{i,k} & \mr_{12}\mP_{1,j} \\ \vzero  & \mr_{22}\mP_{1,j} \end{bmatrix}
\end{align*}
has the unpivoted QR decomposition
\begin{align*}
    \mr \bar{\mP} = \begin{bmatrix} \bar{\mq}_{11} &  \vzero\\ \vzero & \bar{\mq}_{22} \end{bmatrix} \begin{bmatrix} \bar{\mr}_{11} & \bar{\mr}_{12} \\  \vzero & \bar{\mr}_{22} \end{bmatrix}.
\end{align*}
The assumption of non-negative diagonal elements in the upper triangular matrices 
implies $\det(\mr_{11}) = \det(\bar{\mr}_{11})$.
From
\begin{align*}
\bar{\mr}_{11}^{-1} \bar{\mr}_{12} = (\bar{\mq}_{11}^{T} \mr_{11} \mP_{i,k})^{-1} (\bar{\mq}_{11}^{T} \mr_{12} \mP_{1,j}) = \mP_{i,k}^{T} \mr_{11}^{-1} \mr_{12} \mP_{1,j},
\end{align*}
the invariance of the two-norm under multiplication by orthogonal matrices, and the non-negativity
of the diagonal elements follows
\begin{align*}
|\mr_{11}^{-1} \mr_{12}|_{i,j} = |\bar{\mr}_{11}^{-1} \bar{\mr}_{12}|_{k,1},\qquad
\omega_i(\mr_{11}) = \omega_k(\bar{\mr}_{11}),
\qquad \gamma_j(\mr_{22}) = \gamma_1(\bar{\mr}_{22}).
\end{align*}
Thus, the relevant quantities do not change under permutations and subsequent QR decompositions.

\paragraph{Relevant quantities induced by the partitioning of upper triangular matrices}
With $i=k$ and $j=1$,
distinguish\footnote{To increase readability, we sometimes use blank spaces to represent 0 elements.} rows and columns $k$ and $k+1$,
\begin{align*}
    \mr=\begin{bmatrix}
        \begin{array}{c|c}
          \mr_{11} & \mr_{12} \\
          \hline
           & \mr_{22} 
        \end{array}
      \end{bmatrix}
      = \begin{bmatrix}
        \begin{array}{c c |c c}
          \hat{\mr}_{11} & \va & \vb & \hat{\mr}_{12} \\
           & \omega & \beta & \vc^{T} \\
          \hline
           &  & \gamma & \vd^{T} \\
           & & & \hat{\mr}_{22}
        \end{array}
      \end{bmatrix}
\end{align*}
where $\hat{\mr}_{11} \in \real^{(k-1)\times (k-1)}$, $\hat{\mr}_{12} \in \real^{(k-1)\times(p-k-1)}$, 
$\omega>0$ and $\gamma>0$.
Upper triangularity implies the determinant relation
\begin{align}\label{e_B4det}
\det(\mr_{11})=\omega\, \det(\hat{\mr}_{11}).
\end{align}
Looking at the trailing row of $\mr_{11}^{-1}$ and the leading column of $\mr_{22}$,
\begin{align*}
    \mr_{11}^{-1} = \begin{bmatrix} \hat{\mr}_{11}^{-1} & -\frac{1}{\omega}\hat{\mr}_{11}^{-1} \va \\ \vzero & \frac{1}{\omega} \end{bmatrix},\qquad
   \mr_{22} = \begin{bmatrix} \gamma & \vd^{T} \\ \vzero & \hat{\mr}_{22} \end{bmatrix},  
    \end{align*}
gives
    \begin{align*}
1/\|\ve_k^T\mr_{11}^{-1}\|_2=\omega_k(\mr_{11}) = \omega,\qquad
\|\mr_{22}\ve_1\|_2=\gamma_1(\mr_{22}) = \gamma.
\end{align*}
Element $(k, 1)$ of $\mr_{11}^{-1} \mr_{12}$ equals
\begin{align*}
(\mr_{11}^{-1} \mr_{12})_{k,1} = \ve_k^T\begin{bmatrix} \hat{\mr}_{11}^{-1} & -\frac{1}{\omega}\hat{\mr}_{11}^{-1} \va \\ \vzero  & \frac{1}{\omega} \end{bmatrix} \begin{bmatrix} \vb & \hat{\mr}_{12} \\ \beta & \vc^{T} \end{bmatrix}\ve_1=   \begin{bmatrix} \vzero & \frac{1}{\omega}
\end{bmatrix}
\begin{bmatrix}\vb \\ \beta \end{bmatrix} = \frac{\beta}{\omega},
\end{align*}
\paragraph{The action}
Let $\mP$ be the permutation that 
permutes columns $k$ and $k+1$ of $\mr$,
\begin{align*}
\mr\mP = \begin{bmatrix}
        \begin{array}{c c |c c}
          \hat{\mr}_{11} & \vb & \va & \hat{\mr}_{12} \\
           & \beta & \omega & \vc^{T} \\
          \hline
           &  \gamma& 0 & \vd^{T} \\
           & & & \hat{\mr}_{22}
        \end{array}
      \end{bmatrix}
\end{align*}
To return to upper triangular form, perform an unpivoted QR decomposition $\mr\mP=\widetilde{\mq}\widetilde{\mr}$ that zeros out $\gamma$
by rotating rows $k$ and $k+1$.
The resulting triangular matrix $\widetilde{\mr}$ has a leading principal submatrix
\begin{align*}
\widetilde{\mr}_{11}=\begin{bmatrix} \hat{\mr}_{11} & \vb\\\vzero & \sqrt{\beta^2+\gamma^2}\end{bmatrix}
\end{align*}
with the determinant relation
\begin{align*}
\det(\widetilde{\mr}_{11}) = \sqrt{\beta^2+\gamma^2}\det(\hat{\mr}_{11}).
\end{align*}
Combine this with the old determinant relation (\ref{e_B4det})
\begin{align*}
 \frac{\det(\widetilde{\mr}_{11})}{\det(\mr_{11})} & = \frac{\sqrt{\beta^2+\gamma^2}}{\omega}=
 \sqrt{ \left (\frac{\beta}{\omega} \right )^2 + \left (\frac{\gamma}{\omega} \right )^2} \\ &= \sqrt{(\mr_{11}^{-1} \mr_{12})^2_{k,1} + (\gamma_1(\mr_{22})/\omega_k(\mr_{11}))^2}. 
 \end{align*}
\end{proof}

The following proof of Theorem~\ref{t_srrqr} relies on results from \cite[section 3.3]{HoJ91} and \cite[section 3]{GuE96} but without the assumption that $\mchi$ has full column rank.

\begin{lemma}[Proof of Theorem~\ref{t_srrqr}]\label{l_srrqr}
Let $\mchi \in \real^{n \times p}$ with $n > p$ have singular values $\sigma_1 \geq \cdots \geq \sigma_p \geq 0$; and QR decomposition $\mchi=\mq\mr$. Let 
$1 \leq k < p$ so that the leading  $k\times k$ principal submatrix of $\mr$ is non-singular. Then Algorithm~\ref{alg_srrqr} with $f \geq 1$ 
computes a QR decomposition 
\begin{align*}
\mchi\mP=\begin{bmatrix}\mq_1 & \mq_2\end{bmatrix}
\begin{bmatrix}\mr_{11}&\mr_{12}\\\vzero & \mr_{22}\end{bmatrix},
\end{align*}
with singular values 
\begin{align*}
    \sigma_i(\mr_{11}) &\geq \frac{\sigma_i}{\sqrt{1+f^2k(p-k)}}, \qquad 1 \leq i \leq k,\\
    \sigma_j(\mr_{22}) &\leq \sigma_{j+k}\sqrt{1+f^2k(p-k)}, \qquad 1 \leq j \leq p-k.
\end{align*}
Additionally, the elements of $\mr_{11}^{-1}\mr_{12}$ are bounded by
\begin{align*}
    |\mr_{11}^{-1}\mr_{12}|_{i,j} \leq f, \qquad 1 \leq i \leq k \ 1 \leq j \leq p-k.
\end{align*}
\end{lemma}

\begin{proof}
We prove the inequality in the reverse order.

\paragraph{Third inequality} It follows from the observation that
Algorithm~\ref{alg_srrqr} terminates once
\begin{align*}
    |\mr_{11}^{-1}\mr_{12}|_{i,j} \leq \sqrt{|\mr_{11}^{-1}\mr_{12}|^2_{i,j} + (\gamma_j(\mr_{22})/\omega_i(\mr_{11}))^2} \leq f
    \end{align*}
holds for $1 \leq i \leq k$, $1 \leq j \leq p-k$.    

\paragraph{Second inequality}
We scale the leading diagonal block so that it contains the $k$ dominant singular values by  $\alpha \equiv \sigma_1(\mr_{22})/\sigma_k(\mr_{11})$. Extract
a judiciously scaled block-diagonal matrix 
\begin{align*}
    \mr_D \equiv \begin{bmatrix} \alpha \mr_{11} &  \vzero\\ \vzero & \mr_{22}  \end{bmatrix} 
    = \underbrace{\begin{bmatrix} \mr_{11} & \mr_{12} \\ \vzero & \mr_{22} \end{bmatrix}}_{\mr} \underbrace{\begin{bmatrix} \alpha \mi_k & -\mr_{11}^{-1}\mr_{12} \\  \vzero & \mi_{p-k} \end{bmatrix}}_{\mw}.
\end{align*}
where $\alpha \mr_{11}$ contains the $k$ dominant
singular values, because
\begin{align*}
\sigma_k(\alpha\mr_{11})=\alpha\sigma_k(\mr_{11})=\sigma_1(\mr_{22}).
\end{align*}
This means the largest singular value of $\mr_{22}$
is equal to the smallest singular value of $\alpha\mr_{11}$, thus less than or equal to all other singular values of $\alpha\mr_{11}$. Therefore, the trailing block $\mr_{22}$ contains the $p-k$ smallest singular values of $\mr_D$.

The Weyl product inequality \cite[(7.3.13)]{HoJ12} implies
\begin{align}
\label{e_wineq}
    \sigma_{j+k}(\mr_D) = \sigma_j(\mr_{22}) \leq \sigma_{j+k}(\mr)  \|\mw\|_2, \qquad 1 \leq j \leq p-k.
\end{align}
We bound $\|\mw\|_2$,
by bounding the two-norm in terms of the Frobenius norm and, in turn, expressing this as a sum,
\begin{align*}
    \|\mw\|_2^2 &\leq 1+\|\mr_{11}^{-1}\mr_{12}\|_2^2+\alpha^2 \\
    &= 1 + \|\mr_{11}^{-1}\mr_{12}\|_2^2 + \|\mr_{22}\|_2^2 \|\mr_{11}^{-1}\|_2^2 \\
    &\leq 1 + \|\mr_{11}^{-1}\mr_{12}\|_F^2 + \|\mr_{22}\|_F^2 \|\mr_{11}^{-1}\|_F^2 \\
    &= 1 + \sum_{i=1}^k \sum_{j=1}^{p-k} \left ( (\mr_{11}^{-1} \mr_{12} )_{i,j}^2 + (\gamma_j(\mr_{22})/\omega_i(\mr_{11}))^2 \right ) \\
    &\leq 1 + \sum_{i=1}^k \sum_{j=1}^{p-k} f^2 \ = 1 + f^2k(p-k).
\end{align*}
Now substitute
$\|\mw \|_2 \leq \sqrt{1 + f^2k(p-k)}$  into (\ref{e_wineq}).

\paragraph{First inequality}
If $\sigma_1(\mr_{22}) = 0$, then $\mr_{22}=\vzero$ and the first inequality holds.

Thus assume that $\sigma_1(\mr_{22}) > 0$ so that  $\alpha \equiv \sigma_1(\mr_{22})/\sigma_k(\mr_{11})>0$.
Deriving a lower bound for the large singular values requires a slightly different ansatz.  
We scale the trailing block by $1/\alpha$
so that it contains the $p-k$ smallest singular values. Extract a differently scaled block-diagonal matrix,
\begin{align*}
    \mr = \begin{bmatrix} \mr_{11} & \mr_{12} \\ \vzero  & \mr_{22} \end{bmatrix} =   \underbrace{\begin{bmatrix} \mr_{11}  & \vzero \\ \vzero  &  \mr_{22}/\alpha \end{bmatrix}}_{\hat{\mr}_{D}} \underbrace{\begin{bmatrix} \mi_{k} & \mr_{11}^{-1}\mr_{12} \\ \vzero & \alpha \mi_{p-k} \end{bmatrix}}_{\hat{\mw}}.
\end{align*}
where $\mr_{22}/\alpha$ contains the $p-k$ subdominant
singular values, because
\begin{align*}
\sigma_1(\mr_{22}/\alpha)=\sigma_1(\mr_{22})/\alpha=\sigma_k(\mr_{11}).
\end{align*}
This means the smallest singular value of $\mr_{11}$ is equal to the smallest singular value of $\mr_{22}/\alpha$, thus larger or equal to all other singular values of $\mr_{22}/\alpha$. Therefore, the leading block $\mr_{11}$ contains the $k$ largest singular values of $\mr_D$.

An analogous argument as above shows
\begin{align*}
    \sigma_i(\mr) &\leq \sigma_i(\hat{\mr}_D) \|\hat{\mw}\|_{2} =
  \sigma_i(\mr_{11}) \|\hat{\mw}\|_{2}\\
  &\leq 
  \sigma_i(\mr_{11})  \sqrt{1+f^2k(p-k)},
\qquad 1 \leq i \leq k.
\end{align*}
\end{proof}

\section{Supplemental Material}
\label{sec:supp}
We present more details for the models in section~\ref{sec:RealSensDesc}:
Epidemiological (section~\ref{s_mepi}),
cardiovascular tissue (section~\ref{s_cardio}), 
fibrin polymerization (section~\ref{s_wound}), 
and neurological (section~\ref{s_neuro}).
All 
models are represented as coupled systems of ODEs (ordinary differential
equations), and parameter sensitivities are determined 
from their numerical solution via complex-step or finite differences.

\subsection{Epidemiological Models}\label{s_mepi}
\label{subsec:epimodels} We  implemented five (nested) epidemiological compartment models in section \ref{sec:RealSensDesc}
that represent COVID-19 spread among the US population for identifiability anaylsis of the model parameters. 

Figure~\ref{fig:SIRdiags}
displays the different compartments associated with the state variables in each model, and 
the possible transitions from one infection status to another
within a population \cite{TunLe18, Per20}. The parameters above the
arrows represent the transition rates.
 From these diagrams, nonlinear ordinary differential equations for each system can be derived by analogy with leading-order mass action reaction kinetics.
 
For the SIR, SEIR, SVIR, and SEVIR models in Figure \ref{fig:SIRdiags},
the quantity of interest is the number $I(t)$ of infectious individuals at time $t$; and for the COVID model it is $(A+I+H)(t)$.
 
We calibrated the models to the spread of COVID-19 through the US 
based on CDC data and relevant studies.
Table~\ref{tab:epiTab} describes the physical interpretation of each parameter and the average nominal value for generating sensitivities. 

As outlined in \S \ref{sec:physmods}, letting $q_j^*$ represent the nominal value of the $j$th parameter in Table \ref{tab:epiTab}, the algorithms were tested on 10,000 matrices for each model, evaluated at parameter vectors for which the $j$th component is sampled uniformly from the interval  $[0.5q_j^*,1.5q_j^*]$.

\begin{figure}[ht]
    \centering
    \includegraphics[width=\textwidth]{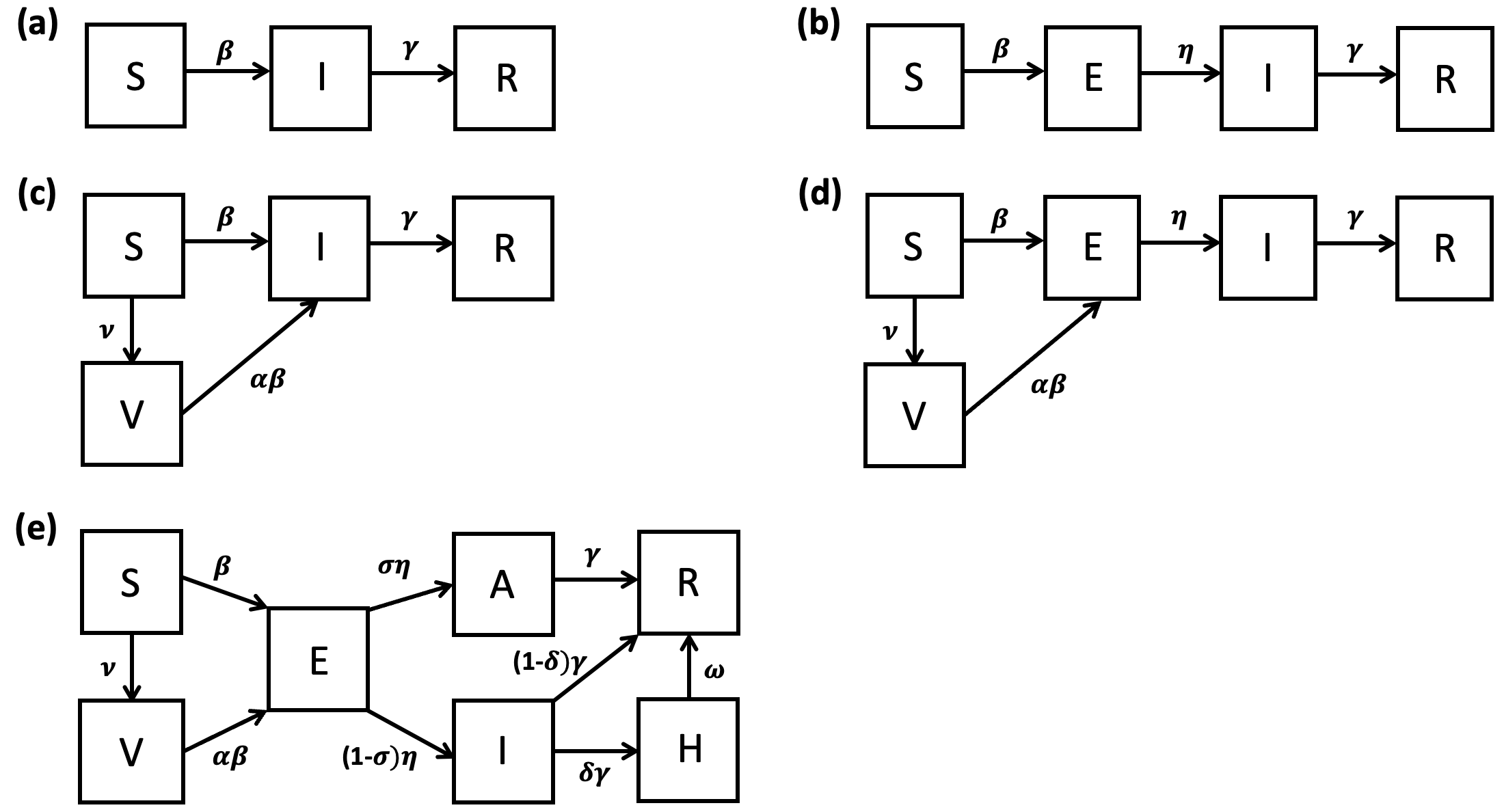}
    \caption{Compartment diagrams for the epidemiological models in section \ref{sec:RealSensDesc} to illustrate the possible transitions from one infection statius to another within a
    population
    for (a) SIR, (b) SEIR, (c) SVIR, (d) SEVIR, and (e) COVID-19 models. }
    \label{fig:SIRdiags}
\end{figure}

\begin{table}[ht]
\label{tab:epiTab}
\centering
\begin{tabular}{|c|c|l|r|}
\hline
Par & Mean & Description & Ref. \\
\hline
$\beta$ & 0.80 & Transmission coefficient & \cite{Ke21} \\
$\eta$ & 0.33 & Rate of progression to infectiousness (following exposure) & \cite{lima20, LiuMag20} \\
$\gamma$ & 0.14 & Rate of progression through infectious stage & \cite{Sanche20, Park20} \\
$\alpha$ & 0.10 & Probability of infection after vaccination & \cite{Per20} \\
$\nu$ & 0.004 & Rate of vaccination & \cite{CDC21} \\
$\sigma$ & 0.35 & Percentage of infected that are asymptomatic & \cite{Sahe21} \\
$\delta$ & 0.05 & Rate of hospitalization for symptomatic infected & \cite{Garg20} \\
$\omega$ & 0.82 & Rate of recovery for hospitalized infections & \cite{Jama21} \\
\hline
\end{tabular}
\caption{Parameter values and physical interpretations in epidemiological models from \S \ref{sec:physmods}}
\end{table}

\subsection{Cardiovascular Tissue Biomechanics (HGO) Model}\label{s_cardio}
The sensitivity matrix corresponding to this model arose from a nonlinear hyperelastic structural model of the vessel wall for a large pulmonary artery in the context of ex vivo biomechanical experiments. 
A two-layer, anisotropic vessel wall model was developed, within the general framework of the Holzapfel-Gasser-Ogden (HGO) model \cite{HGO}, and systematically reduced with identifiabilty techniques rooted in the scaled sensitivity matrix. 
The quantity of interest  was a hybrid normalized residual vector amalgamating data measuring lumen area and wall thickness changes with increasing fluid pressure.

This data set arises from ex vivo biomechanical testing of coupled flow and deformation for left pulmonary arteries excised from normal and hypertensive mice. 
This model contains 16 model parameters: 8 are fixed based on values in the literature or information from the experiments, while the remaining 8 are estimated via systematic model reduction in the context of an inverse problem \cite{Hai}.

Results of the systematic model reduction in \cite{Hai} are consistent with the value $k=5$ for HGO in Table 4.1.

\subsection{Fibrin Polymerization Model for Wound Healing Applications}\label{s_wound}
Motivated by a wound healing application,
this 
is a biochemical reaction kinetics model for in vitro fibrin  polymerization, mediated by the enzyme thrombin, and  \cite{Pearce21}. 
The $46\times 11$ sensitivity matrix
represents 46 time points for the concentration of fibrin matrix,
i.e.~in-vitro clots;
and 11 parameters that represent
reaction rates for the associated
 biochemical reaction species. The
 parameters are chosen from 
 the last row of \cite[Table 1]{Pearce21} for a mathematical model of hemostasis, the first stage of wound healing during which fibrin (extracellular) matrix polymerization occurs. 

The corresponding system of ODEs is based on first-order reaction kinetics, analogous to the mass-action assumptions for the epidemiological compartment models in 
section~\ref{s_mepi}.
The systematic identifiability analysis and model reduction for the inverse problem in \cite{Pearce21} are consistent with the value $k=6$ under the "Wound"
model in Table~\ref{tab:RealTab}.

\subsection{Neurological Model}\label{s_neuro}
This complex  model consists of
a system of nonlinear ODEs \cite{hart2019} that 
quantify the neurovascular coupling (NVC) response, and the local changes in vascular resistance due to neuronal activity \cite{hart2019}. The
state variables represent different components of the human brain, while the parameters represent the ion channels and metabolic signalling among them.

The  sensitivity matrix is the
largest and most ill-conditioned sensitivity matrix in Table \ref{tab:RealTab}, along with
the largest number 
of state variables, parameters $p=175$, 
and  observations $n=200$. 

\section{Dynamical systems for adversarial CSS matrices}
\label{sec:supp3}

Given an adversarial CSS matrix $\mchi\in\real^{n\times p}$ with $n\geq p$, we construct a dynamical system whose sensitivity matrix is identical to $\mchi$.

Let $\ms = \mU\msig\mv^\top$ have
have a thin SVD as in (\ref{e_SVDS}) and distinguish 
the columns of the singular vector matrices,
\begin{align*}
\mU=\begin{bmatrix}\vu_1 & \ldots & \vu_p\end{bmatrix}\in\real^{n\times p},
\qquad
\mv=\begin{bmatrix}\vv_1 & \ldots & \vv_p\end{bmatrix}\in\real^{p\times p}.
\end{align*}
Pick some vector $\vq\in\real^p$,
we are going to construct a system of ODEs parameterized by~$\vq$.

To this end, let
\begin{align*}
\mlam \equiv \diag
\begin{pmatrix}\lambda_1  & \cdots & \lambda_p \end{pmatrix}\in\real^{p\times p}
\end{align*}
be a diagonal matrix yet to be specified. 
Denote by $\vx(t)\in\real^p$ the state vector 
and by $\vy(t)$ the observation vector,
and combine everything 
into the initial value problem
\begin{align}\label{e_dyn}
\begin{split}
\frac{d\vx}{dt} &=  \> \mlam\vx,
\qquad \vx(0) = \mv^T\vq,\\ 
\vy(t) &=  \> \mU\vx(t)
\end{split}
\end{align} 
with solution $\vx(t) = \exp(t\mlam) \mv^T\vq$.
Since $\mlam$ is diagonal, the observation equals
 \begin{align*}
 \vy(t) = \mU \exp(t\mlam) \mv^T\vq
 = \sum_{j=1}^p{\vu_j e^{t \lambda_j}\vv_j^T\vq}. 
 \end{align*}
 As in section~\ref{s_rsd1}, differentiate $\vy$ with respect to $\vq$, and then evaluate at time $t=\tau>0$. The
rows of the resulting sensitivity matrix equal 
\[ \frac{\partial y_i}{\partial \vq} =  
\sum_{j = 1}^p{u_{ij} \,e^{\tau \lambda_j}\,\vv_j^T},
\qquad 1 \leq i \leq n. \]
Set $\sigma_j=e^{\tau \lambda_j}$ so that $\lambda_j=\frac{1}{\tau}\ln{\sigma_j}$, $1\leq j\leq p$. Then the sensitivity matrix at time $t=\tau$ is
\[ \ms(t;\vq) = \sum_{j = 1}^p \vu_j \sigma_j\,\vv_j^T = 
\mU \msig \mv^T.\]
Therefore, the dynamical system (\ref{e_dyn})
has the desired sensitivity matrix $\mchi$ at time~$\tau$.

\bibliography{parameter}
\end{document}